\documentclass[11pt]{article}
\usepackage{amstext, amsmath,latexsym,amsbsy,amssymb,amsmath}

\usepackage[version=4]{mhchem}

\usepackage{enumitem}
\newtheorem{definition}{Definition}[section]

\newtheorem{remark}{Remark}[section]
\newenvironment{proof}{{\bf Proof\ }}{\QED\\}

\numberwithin{equation}{section}
\newtheorem{theorem}{Theorem}[section]

\newcommand{\QED}{\hspace*{\fill}\rule{2.5mm}{2.5mm}}

\usepackage{amssymb}\usepackage{graphicx}
\usepackage{tikz}

\usepackage{color}
\newcommand\qed{\hfill$\sqcap\kern-7.5pt\hbox{$\sqcup$}$}

\newcommand{\beqn}{\begin{equation}}
\newcommand{\eeqn}{\end{equation}}
\newcommand{\bear}{\begin{eqnarray}}
\newcommand{\eear}{\end{eqnarray}}
\newcommand{\bean}{\begin{eqnarray*}}
\newcommand{\eean}{\end{eqnarray*}}

\allowdisplaybreaks

\begin{document}
\title{A  reaction network approach to the convergence to equilibrium of quantum Boltzmann equations for Bose gases}
\author{
\\\\Gheorghe Craciun *, Minh-Binh Tran **\\
*Department of Mathematics and Department of Biomolecular Chemistry \\
University of Wisconsin-Madison\\
email: craciun@wisc.edu\\
**Department of Mathematics\\ Southern Methodist University\\ email: minhbinht@mail.smu.edu}

\maketitle

\centerline{\it Dedicated to the  60th birthday of Professor Enrique Zuazua}
\begin{abstract} When the temperature of a trapped Bose gas is below the Bose-Einstein transition temperature and above absolute zero, the gas is composed of two distinct components: the Bose-Einstein condensate and the 
cloud of thermal excitations. The dynamics of the excitations can be  described by  quantum Boltzmann models. We establish a connection between quantum Boltzmann models and chemical reaction networks. We prove that  the  discrete differential equations for these quantum Boltzmann models converge to an equilibrium point. Moreover, this point is unique for all initial conditions that satisfy the same conservation laws. In the proof, we then employ a toric dynamical system approach, similar to the one used to prove the global attractor conjecture, to study the convergence to equilibrium of  quantum kinetic equations, derived in \cite{tran2020boltzmann,tran2020boltzmann2}.  
\end{abstract}{\bf Keywords:} quantum Boltzmann equation, dynamical systems, bosons, Bose-Einstein condensate, rate of convergence to equilibrium, global attractor conjecture, mass-action kinetics, power law
systems, biochemical networks, Petri net.
{
\tableofcontents

\section{Introduction}\label{Intro}
 Several years after the invention of the Boltzmann--Nordheim equation, which is  the quantum version of the classical Boltzmann one, to describe the evolution of dilute quantum gases (cf.  \cite{Nordheim,UehlingUhlenbeck}), a renewal in the  kinetic theory of bosons has started by the pioneering work  of Kirkpatrick and Dorfman \cite{KD1,KD2}. 
This work of Kirkpatrick and Dorfman was later extended by Zaremba, Nikuni and Griffin \cite{ZarembaNikuniGriffin:1999:DOT}, in which the full coupling system of a quantum Boltzmann equation for the density function of the normal fluid/thermal cloud and a Gross--Pitaevskii equation for the wavefunction of the BEC has been introduced. In an independent work, the same model was  derived by Pomeau,  Brachet, M\'{e}tens and Rica in \cite{PomeauBrachetMetensRica}. We prefer to \cite{GriffinNikuniZaremba:BCG:2009,pomeau2019statistical} for further discussions on the topic. In the models by Zaremba, Nikuni and Griffin and  Pomeau,  Brachet, M\'{e}tens and Rica , there are two type of collisional processes. 
 \begin{itemize}
 \item The $1\leftrightarrow2$ interactions between the condensate and the excited atoms, described by the $C_{12}$ collision operator.
 \item The $C_{22}$ collision operator describes The $2\leftrightarrow2$ interactions between the excited atoms themselves, described by the $C_{22}$ collision operator. 
 \end{itemize}
A third collisional process, previously missing, was proposed by Reichl and Gust 
 \cite{ReichlGust:2012:CII,reichl2019kinetic}. This process takes into account $1{\leftrightarrow}3$ type collisions between the excitations and is described by  the collision operator ${C}_{31}$.
However, the derivation of the new collision operator ${C}_{31}$ was very complicated, since  it involves the computations of around $40000$ individual terms.  As a result, a    concise mathematical justification for the existence of the missing collision operator ${C}_{31}$ had been open for many years, and has been solved only until recently in \cite{tran2020boltzmann}.

 The spatial homogeneous kinetic equation for the evolution of the density function $f(t,p)$ of the thermal cloud, derived in \cite{tran2020boltzmann}[Section I],  takes the form

\begin{equation}
\label{KineticFinal}
\begin{aligned}
& \partial_t f(p) \ =  \ C_{12}[f](p) \  +  \  C_{22}[f](p)  \ + \ C_{31}[f](p),
\end{aligned}
\end{equation}
in which the forms of $C_{12}$, $C_{22}$, $C_{31}$ are given explicitly below
\begin{equation}
\begin{aligned}\label{C12Discrete}
& C_{12}[f](t,p) \ =   \  4\pi {\frac{g^2n}{V}} \sum_{p_1,p_2,p_3\ne 0}(\delta(p-p_1)-\delta(p-p_2)\\
&-\delta(p-p_3)) \\
& \times \delta(\omega(p_1)-\omega(p_2)-\omega(p_3))(K^{12}_{123})^2\delta(p_1-p_2-p_3)\\
& \times \Big[f(p_2)f(p_3)(f(p_1)+1)-f(p_1)(f(p_2)+1)(f(p_3)+1)\Big],
\end{aligned}
\end{equation}
\begin{equation}
\begin{aligned}\label{C22Discrete}
& C_{22}[f](t,p) \ =   \  \frac{g^2\pi}{V^2}  \sum_{p_1,p_2,p_3,p_4\ne 0}(\delta(p-p_1)+\delta(p-p_2)\\
& -\delta(p-p_3)-\delta(p-p_4))(K^{22}_{1234})^2\\
&\times \delta(p_1+p_2-p_3-p_4)\delta(\omega(p_1)+\omega(p_2)-\omega(p_3)-\omega(p_4))\\
&\times \Big[f(p_3)f(p_4)(f(p_2)+1)(f(p_1)+1)\\
& -f(p_1)f(p_2)(f(p_3)+1)(f(p_4)+1)\Big],
\end{aligned}
\end{equation}
and
\begin{equation}
\begin{aligned}\label{C31Discrete}
& C_{31}[f](t,p)  = \ {\frac{3g^2\pi}{V}} \sum_{p_1,p_2,p_3,p_4\ne 0}(\delta(p-p_1)-\delta(p-p_2)\\
&-\delta(p-p_3)-\delta(p-p_4))\\
&\times (K^{31}_{1234})^2\delta(p_1-p_2-p_3-p_4)\\
&\times\delta(\omega(p_1)-\omega(p_2)-\omega(p_3)-\omega(p_4))\\
&\times \Big[f(p_3)f(p_4)f(p_2)(f(p_1)+1)\\
&-f(p_1)(f(p_2)+1)(f(p_3)+1)(f(p_4)+1)\Big],
\end{aligned}
\end{equation}
in which $n$ is the density of the condensate, $t\in\mathbb{R}_+$ is the time variable, $p\in(\mathbb{Z}/L)^d\backslash\{O\}$ is the $d$-dimensional non-zero momentum variable, $V$ is proportional to the volume of the periodic box  $\left[-\frac{L}{2},\frac{L}{2}\right]^d$, $m$ is the particle mass, $\omega$ is the Bogoliubov dispersion relation defined as \begin{equation}
\label{BogDispersion}
\omega_p \ = \ \left[\frac{gn}{m}p^2 \ + \ \left(\frac{p^2}{2m}\right)^2\right]^\frac12
\end{equation} and $g$ is the interacting constant. We have normalized the Plank constant to be $1$.
In the above collision operators, the kernels are defined as follows
\begin{equation}
\label{KernelC12}\begin{aligned}
K_{1,2,3}^{1,2}  = & \ u_{p_1}u_{p_2}u_{p_3}-v_{p_1}v_{p_2}v_{p_3}-u_{p_1}u_{p_2}v_{p_3}\\
&+v_{p_1}v_{p_2}u_{p_3}-u_{p_1}v_{p_2}u_{p_3}+v_{p_1}u_{p_2}v_{p_3},\end{aligned}
\end{equation}
\begin{equation}
\label{KernelC22}\begin{aligned}
K^{2,2}_{1,2,3,4} \ = & \ u_{p_1}u_{p_2}u_{p_3}u_{p_4}+u_{p_1}v_{p_2}u_{p_3}v_{p_4}+u_{p_1}v_{p_2}v_{p_3}u_{p_4}\\
& +v_{p_1}u_{p_2}v_{p_3}u_{p_4}+v_{p_1}u_{p_2}u_{p_3}v_{p_4}+v_{p_1}v_{p_2}v_{p_3}v_{p_4},\end{aligned}
\end{equation}
and
\begin{equation}
\label{KernelC31}\begin{aligned}
K^{3,1}_{1,2,3,4} \ = & \ 2\Big[u_{p_1}u_{p_2}v_{p_3}u_{p_4}+v_{p_1}v_{p_2}u_{p_3}v_{p_4}\Big],\end{aligned}
\end{equation}
with $u_p$ and $v_p$ being defined as
\begin{equation}
\label{BogConstant}
u_p, v_p \ = \ \left(\frac{\epsilon_p +gn}{2\omega_p}\pm\frac12\right)^\frac12.
\end{equation}
In the setting of \cite{tran2020boltzmann}, we could fix $n$ as a constant, under the assumption that the thermal could fraction is quite small, in comparison to the condensate. Moreover, in the sum on the momenta $\sum_{p\ne 0}$, the origin is removed   due to  the fact that the condensate has been factored out  in the Bogoliubov diagonalization (cf. \cite{tran2020boltzmann,tran2020boltzmann2}).

\begin{remark}
	
	As it has been discussed in  \cite{tran2020boltzmann}, the BEC is in a cubic box with periodic boundary conditions, the quantum Boltzmann equation is then in the discrete form. In order for the conservations of momentum and energy to be satisfied, the following system needs to have solutions on the lattice
\begin{equation}\label{ConservationSystem}
\begin{aligned}
& p_1=p_2+p_3+p_4, \ \ \omega(p_1)=\omega(p_2)+\omega(p_3)+\omega(p_4),
\\
& p_1'+p_2'=p_3'+p_4', \ \ \omega(p_1')+\omega(p_2')=\omega(p_3')+\omega(p_4'),\\
& p_1''=p_2''+p_3'', \ \ \omega(p_1'')=\omega(p_2'')+\omega(p_3'').
\end{aligned}
\end{equation}
At the first sign, the system does have solutions due to the complicated form of the Bogoliubov dispersion relation \eqref{BogDispersion}. However, it has been pointed out in \cite{tran2020boltzmann,tran2020boltzmann2} that  when the temperature of the system is lower but closed to the Bose-Einstein condensation transition temperature, the Bogoliubov dispersion relation can be replaced by the Hatree-Fock energy ($\omega(p)\approx c|p|^2$). In this regime,  the two collision operators $C_{12}$ and $C_{22}$ dominate the collisional processes. The contribution of third collision operator $C_{31}$ becomes non-trivial when both $u_p$ and $v_p$ are large, corresponding to significantly low temperatures.  In this low temperature regime, the excitations are phonon-like and the Bogoliubov dispersion relation  \eqref{BogDispersion}    can be replaced by the phonon dispersion relation \eqref{E3a}. The replacement of  \eqref{BogDispersion}    by \eqref{E3a} guarantees the existence of solutions to \eqref{ConservationSystem}, and thus, the conservation laws are satisfied.
\end{remark}
{\bf Simplified Quantum Boltzmann model of the thermal cloud.} In our work, we try to provide a deeper understanding of the property of the system derived in \cite{tran2020boltzmann} by studying a simplified version of it. If we denote 
$$f_1=f(t,p_1), f_2=f(t,p_2), f_3=f(t,p_3), f_4=f(t,p_4),$$
then our simplified system for $f_1$ writes
\begin{eqnarray}\label{QuantumBoltzmannLinda}
\frac{\partial f_1}{\partial t}&=&C_{12}[f_1]+C_{22}[f_1]+C_{13}[f_1],\end{eqnarray}

where
\begin{eqnarray}
C_{22}[f_1]&:=&\int_{\mathbb{R}^{9}} K^{22}_{p_1, p_2, p_3,p_4}\delta(p_1+p_2-p_3-p_4)\delta(\mathcal{E}_{p_1}+\mathcal{E}_{p_2}-\mathcal{E}_{p_3}-\mathcal{E}_{p_4})\\\label{QuantumBoltzmann2}
& &\times[(1+f_1)(1+f_2)f_3f_4-f_1f_2(1+f_3)(1+f_4)]dp_2dp_3dp_4,\\\nonumber
C_{12}[f_1]&:=&\int_{\mathbb{R}^{6}} K^{12}_{p_1, p_2, p_3}\delta(p_1-p_2-p_3)\delta(\mathcal{E}_{p_1}-\mathcal{E}_{p_2}-\mathcal{E}_{p_3})\\\label{QuantumBoltzmann3}
& &\times[(1+f_1)f_2f_3-f_1(1+f_2)(1+f_3)]dp_2dp_3\\\nonumber
&&-2\int_{\mathbb{R}^{6}} K^{12}_{p_1, p_2, p_3}\delta(p_2-p_1-p_3)\delta(\mathcal{E}_{p_2}-\mathcal{E}_{p_1}-\mathcal{E}_{p_3})\\\nonumber
& &\times[(1+f_2)f_1f_3-f_2(1+f_1)(1+f_3)]dp_2dp_3,
\end{eqnarray}
{and}

\begin{equation}\label{QuantumBoltzmannLinda13}
\begin{aligned}
C_{13}[f_1]\ =&\int_{\mathbb{R}^{3\times 3}}  K^{13}_{p_1, p_2, p_3,p_4}\delta(p_1-p_2-p_3-p_4)\delta(\mathcal{E}_{p_1}-\mathcal{E}_{p_2}-\mathcal{E}_{p_3}-\mathcal{E}_{p_4})\\
&\ \times[(1+f_1)f_2f_3f_4-f_1(1+f_2)(1+f_3)(1+f_4)]dp_2dp_3dp_4\\
&\ -3\int_{\mathbb{R}^{3\times 3}} K^{13}_{p_1, p_2, p_3,p_4}\delta(p_2-p_1-p_3-p_4)\delta(\mathcal{E}_{p_2}-\mathcal{E}_{p_1}-\mathcal{E}_{p_3}-\mathcal{E}_{p_4})\\
&\ \times[(1+f_2)f_1f_3f_4-f_2(1+f_1)(1+f_3)(1+f_4)]dp_2dp_3dp_4,
\end{aligned}
\end{equation}
The quantities  $ K^{22}_{p_1, p_2, p_3,p_4},  K^{12}_{p_1, p_2, p_3}\geq 0$ are the  collision kernels, which are radially symmetric, and symmetric with respect to the permutation of $p_1$, $p_2$, $p_3,$ and $p_4$:
\begin{align*}
&~~ K^{22}_{p_1, p_2, p_3,p_4}= K^{22}_{|p_1|,|p_2|,|p_3|,|p_4|}= K^{22}_{|p_2|,|p_1|,|p_3|,|p_4|}= K^{22}_{|p_3|,|p_2|,|p_1|,|p_4|}\\
&~~= K^{22}_{|p_4|,|p_2|,|p_3|,|p_1|}= K^{22}_{|p_1|,|p_3|,|p_2|,|p_4|}=\mathcal K^{22}_{|p_1|,|p_4|,|p_3|,|p_2|}= K^{22}_{|p_1|,|p_2|,|p_4|,|p_3|},
\end{align*}
and
\begin{align*}
&~~ K^{12}_{p_1, p_2, p_3}= K^{12}_{|p_1|,|p_2|,|p_3|}= K^{12}_{|p_2|,|p_1|,|p_3|}= K^{12}_{|p_3|,|p_2|,|p_1|}= K^{12}_{|p_1|,|p_3|,|p_2|},
\end{align*}
where $|p|$ denotes the length of the vector $p$.  and $\mathcal K^{13}_{p_1, p_2, p_3,p_4}$ is  positive, radially symmetric, and symmetric with respect to the permutation of $p_1$, $p_2$, $p_3$, $p_4$ 
\begin{align*}
&~~ K^{13}_{p_1, p_2, p_3,p_4}= K^{13}_{|p_1|,|p_2|,|p_3|,|p_4|}= K^{13}_{|p_2|,|p_1|,|p_3|,|p_4|}= K^{13}_{|p_3|,|p_2|,|p_1|,|p_4|}\\
&~~= K^{13}_{|p_4|,|p_2|,|p_3|,|p_1|}= K^{13}_{|p_1|,|p_3|,|p_2|,|p_4|}=\mathcal K^{13}_{|p_1|,|p_4|,|p_3|,|p_2|}= K^{13}_{|p_1|,|p_2|,|p_4|,|p_3|}.
\end{align*}
We make a further simplification by supposing that the temperature is very low compared to the Bose-Einstein critical temperature. As a result, the energy $\mathcal{E}_p=\mathcal{E} (p)$ is given by the phonon dispersion law (cf. \cite{ReichlBook}):
\bear
&&\mathcal{E} (p)=c |p|, \,\,\,c=\sqrt{\frac {gn_c} {m}}. \label{E3a}
\eear
{\bf Reaction networks and a toric dynamical system approach for the relaxation to equilibrium problem.}
 The study of the relaxation of BECs to thermodynamic equilibrium has  played very important role in the theory of Bose gases \cite{ReichlGust:2013:TTF,ReichlGust:2013:RRA,ReichlGust:2012:CII,GriffinNikuniZaremba:2009:BCG,ZarembaNikuniGriffin:1999:DOT}. Our main tool is to convert these equations into chemical reaction systems and use an extension of the theory of toric dynamical systems (cf. \cite{MR2561288}).  
\bigskip
\\ In general, there is great interest in understanding the qualitative behavior of deterministically modeled chemical reaction systems, including the existence of positive equilibria, stability properties of equilibria, and the non-extinction, or persistence, of species, which are the constituents of these systems \cite{craciun2019polynomial,yu2018mathematical,Feinberg:1972:CBI,Feinberg:1995:TEA,Horn:1974:TDO,Anderson:2001:APG,AngeliLeenheerSontag:2011:PRF,MR3199409, MR2604624,MR2734052,MR2561288}.  Toric dynamical systems -- originally called complex-balanced systems (cf. \cite{MR2561288,HornJackson:1972:GMA}) -- are models used to describe an important class of chemical kinetics. The complex-balanced condition was first introduced by Boltzmann  \cite{Boltzmann} for modeling collisions in kinetic gas theory. Based on this condition, it was shown by Horn and Jackson \cite{HornJackson:1972:GMA,Horn:1972:NAS,Feinberglecture,Gunawardena}  that a complex-balanced system has a unique locally stable equilibrium within each linear invariant subspace.  To underline the tight connection to the algebraic study of toric varieties, the name ``toric dynamical system'' was proposed in \cite{MR2561288}. The most important problem in the theory of toric dynamical systems is the Global Attractor Conjecture, which says that the complex balanced equilibrium of a toric dynamical system is a globally attracting point within each linear invariant subspace. This global attractor question is strongly related to the convergence to equilibrium problem in the study of kinetic equations.  A proof to the Global Attractor Conjecture for small dimensional systems has been supplied in \cite{CraciunNazarovPantea:2013:PAP}, for strongly connected networks in \cite{Anderson:2001:APG}, and a complete proof has been proposed in  \cite{Craciun:2015:TDI}. 
\bigskip\\
{ Our goal is to use the tools developed in \cite{CraciunNazarovPantea:2013:PAP,Craciun:2015:TDI} to prove the relaxation to equilibrium of Discrete Velocity Models of  a  model of \eqref{QuantumBoltzmannLinda}, whose collision operator is $C_{12}$. Similarly, we will prove the relaxation to equilibrium of another model of  \eqref{QuantumBoltzmannLinda}, whose collision operator is $C_{12}+C_{22}$, and modified quantum Boltzmann model of the thermal cloud \eqref{QuantumBoltzmannLinda}, whose collision operator is $C_{12}+C_{22}+C_{13}$.} A related approach for the study of acoustic wave turbulence has been used in \cite{CraciunSmithBoldyrevBinh}. Let us also mention that some mathematical results of similar kinetic models have been obtained  in \cite{AlonsoGambaBinh,ArkerydNouri:AMP:2013,ArkerydNouri:2015:BCI,ArkerydNouri:2012:BCI,bernhoff2015half,bernhoff2017boundary,cai2018spatially,EscobedoVelazquez:2015:FTB,JinBinh,Lu:2004:OID,Lu:2005:TBE,Lu:2013:TBE,Binh9,Spohn:2010:KOT,germain2020optimal,ToanBinh,SofferBinh1,SofferBinh2}. 
\bigskip\\
%
%
 The plan of our paper is the following:
\begin{itemize}
\item In section \ref{Sec:Q12}, we show that the discrete version of a simplified version of  \eqref{QuantumBoltzmannLinda}, that contains only $C_{12}$, could be rewritten as a chemical reaction network. By using an approach inspired by the theory of toric dynamical system, we prove in Theorem 2.1. that the solution of the discrete version of a simplified version of  \eqref{QuantumBoltzmannLinda}, that contains only $C_{12}$, converges to the equilibrium exponentially in time.
\item In section \ref{Sec:Qnm}, we generalize Theorem \ref{TheoremQ12} to  collision operators of the forms $C_{13}$ and $C_{22}$. We prove that the solutions of the discrete versions of these equations, associated with the collision operators $C_{13}$ and $C_{22}$ converge to equilibria exponentially in Theorems 3.1 and \ref{TheoremQnm}. In the case of $C_{22}$, we consider a one-dimensional version of the model.
\item In Theorem \ref{TheoremSumQnm} of Section 4, we extend Theorem \ref{TheoremQnm} to   a simplified version of  \eqref{QuantumBoltzmannLinda}, that contains only $C_{12}+C_{22}$, and the modified quantum Boltzmann model of the thermal cloud \eqref{QuantumBoltzmannLinda}, that contains only $C_{12}+C_{22}+C_{31}$. 
\end{itemize}
\section{A reaction network approach for the case of $C_{12}$}\label{Sec:Q12}
\subsection{The dynamical system associated to $C_{12}$}
As mentioned in the introduction, the model derived from physics to describe the system that couples BEC-excitations at very low temperature  is the  discrete version of a simplified version of  \eqref{QuantumBoltzmannLinda}, that contains only $C_{12}$, described below.
\\\\ Let $\mathcal{L}_R$ denote the lattice of integer points 
$$\mathcal{L}_R=\{p\in\mathbb{Z}^3, |p|<R\}.$$
The discrete version of the simplified version of  \eqref{QuantumBoltzmannLinda}, that contains only $C_{12}$, reads 
\begin{equation}\label{DiscreteQuantum}\begin{aligned}
\dot{f}_{p_1}=
&\sum_{\substack{p_2,p_3\in\mathcal{L}_R,\\ p_1-p_2-p_3=0, \\ \mathcal{E}(p_1)-\mathcal{E}(p_2)-\mathcal{E}(p_3)=0}}{K}^{12}_{p_1,p_2,p_3}\left\{(f_{p_1}+1)f_{p_2}f_{p_3}-f_{p_1}(f_{p_2}+1)(f_{p_3}+1)\right\}\\
&\quad -2\sum_{\substack{p_2,p_3\in\mathcal{L}_R,\\ p_1+p_2-p_3=0,\\ \mathcal{E}(p_1)+\mathcal{E}(p_2)-\mathcal{E}(p_3)=0}}{K}^{12}_{p_1,p_2,p_3}\left\{(f_{p_3}+1)f_{p_1}f_{p_2}-f_{p_3}(f_{p_1}+1)(f_{p_2}+1)\right\},~~\end{aligned}
\end{equation}
for all $p_1$ in $\mathcal{L}_R$, where  $\mathcal{E}(p)$ is defined in \eqref{E3a}.
\subsection{Decoupling the quantum Boltzmann equation associated to $C_{12}$}
 Note that when $p_1=0$, $\mathcal{K}^{12}_{p_1,p_2,p_3}$ is also $0$, and therefore, we get 
\begin{eqnarray}\label{DiscreteQuantumIndex0}
\dot{f}_0=0,
\end{eqnarray}
which says that $f_0(t)$ is a constant for all time $t$. Moreover, $f_{p_1}$ does not depend on $f_0$ for all $p_1\ne 0$. Therefore, without loss of generality, we can suppose that $f_0(0)=0$, which leads to $f_0(t)=0$ for all $t$.
\\ Taking into account the fact  $\mathcal{E}(p)=c|p|$, note that if $p_1,p_2,p_3\in\mathcal{L}_R$ are different from $0$ and $p_3=p_1+p_2$ and $|p_3|=|p_1|+|p_2|$ (like in the second sum of \eqref{DiscreteQuantum}), then $p_1,p_2,p_3$ must be collinear and on the same side of the origin. Therefore, we infer that there exists a vector $P$  and  $k_1$, $k_2$, $k_3>0$, $k_1,k_2,k_3\in\mathbb{Z}$ such that 
$$p_1=k_1 P;~~~p_2=k_2 P;~~~p_3=k_3 P,~~~k_1+k_2=k_3.$$  
Since $\mathcal{L}_R$ is bounded, it follows that $k_1,k_2,k_3$ belong to a finite set of integer indices $\mathbb{I}=\{1,\dots,I\}$.
Arguing similarly for the first sum in \eqref{DiscreteQuantum}, we deduce that \eqref{DiscreteQuantum} is equivalent with the following system for $k_1\in\mathbb{I}$
\begin{equation}\label{DiscreteQuantum1DVersion2}\begin{aligned}
\dot{f}_{Pk_1}
=&\quad\sum_{\substack{k_2,k_3\in
\mathbb{I},\\ k_1-k_2-k_3=0}}{K}^{12}_{Pk_1,Pk_2,Pk_3}\left\{(f_{Pk_1}+1)f_{Pk_2}f_{Pk_3}-f_{Pk_1}(f_{Pk_2}+1)(f_{Pk_3}+1)\right\}\\
&\quad-{2}\sum_{\substack{k_2,k_3\in
\mathbb{I},\\ k_1+k_2-k_3=0}}{K}^{12}_{Pk_1,Pk_2,Pk_3}\left\{(f_{Pk_3}+1)f_{Pk_1}f_{Pk_2}-f_{Pk_3}(f_{Pk_1}+1)(f_{Pk_2}+1)\right\}.
\end{aligned}
\end{equation}
Note that the system of equations \eqref{DiscreteQuantum1DVersion2} shows a {\it decoupling} of the system of equations \eqref{DiscreteQuantum} along a ray $\{kP_0\}$ with $k>0$ (see Figure 1). As a consequence, it is sufficient to study the system of equations \eqref{DiscreteQuantum1DVersion2} for a fixed value of $P_0$, instead of the system of equations \eqref{DiscreteQuantum}.

If we denote $f_{k_1P_0}$ by $\bar{f}_{k_1}$ (with $k_1\in \mathbb{I}$) and ${K}^{12}_{k_1P_0,k_2P_0,k_3P_0}$  by $\mathcal{K}^{12}_{k_1,k_2,k_3}$, we obtain the following new system for the ray $\{k_1P_0|k_1>0\}$:
\begin{equation}\label{DiscreteQuantum1D}\begin{aligned}
\dot{\bar{f}}_{k_1}=&\quad\sum_{\substack{k_2,k_3\in
\mathbb{I},\\ k_1=k_2+k_3}}\mathcal{K}^{12}_{k_1,k_2,k_3}\{(\bar{f}_{k_1}+1)\bar{f}_{k_2}\bar{f}_{k_3}-\bar{f}_{k_1}(\bar{f}_{k_2}+1)(\bar{f}_{k_3}+1)\}
\\
&\quad-2\sum_{\substack{k_2,k_3\in
\mathbb{I},\\ k_1+k_2=k_3}}\mathcal{K}^{12}_{k_1,k_2,k_3}\{(\bar{f}_{k_3}+1)\bar{f}_{k_1}\bar{f}_{k_2}-\bar{f}_{k_3}(\bar{f}_{k_1}+1)(\bar{f}_{k_2}+1)\}, ~\forall k_1\in\mathbb{I}.\end{aligned}
\end{equation}
A simple calculation leads to the following {\it conservation of energy}
\begin{equation}\label{MassConservation1}
\sum_{k=1}^I k\dot{\bar{f}}_k=0,
\end{equation}
or equivalently
\begin{equation}\label{MassConservation2}
\sum_{k=1}^I k{\bar{f}_k}=\mbox{const}.
\end{equation}
\begin{figure}
  \centering
  \includegraphics[width=.9\linewidth]{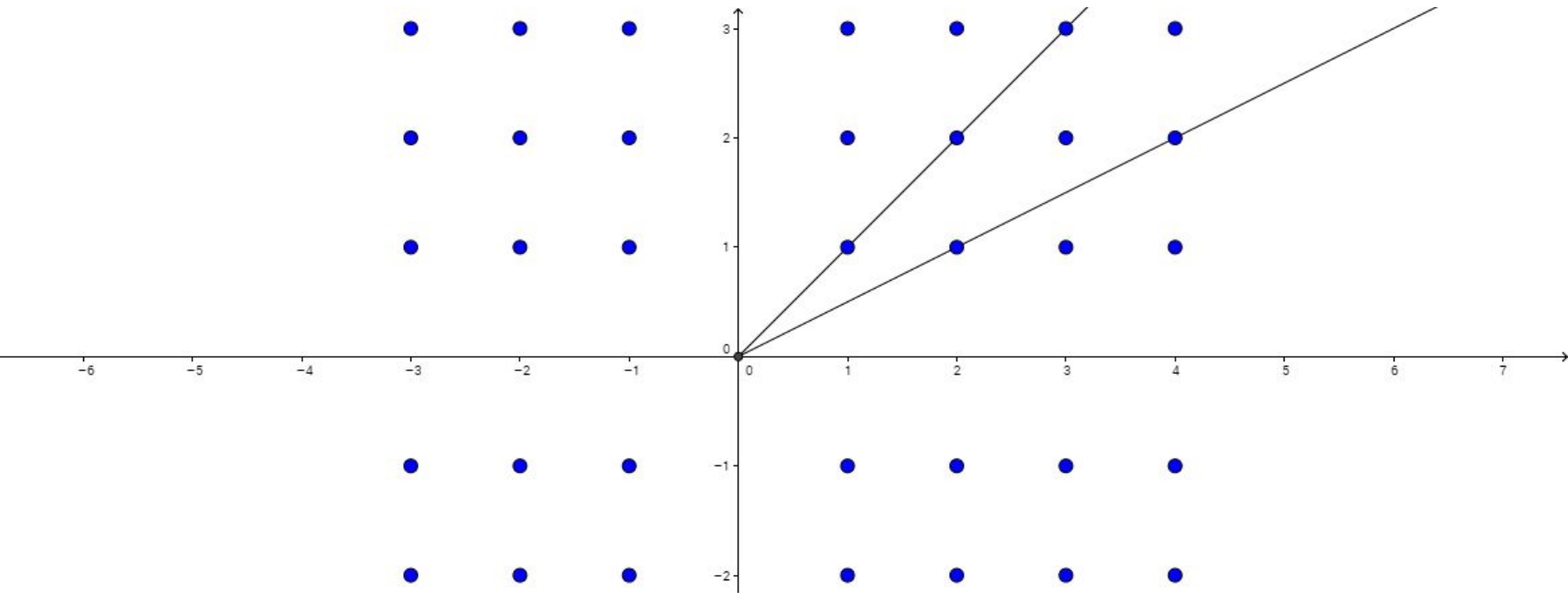}
\caption{We decouple the system \eqref{DiscreteQuantum} into rays}
\label{Fig2}
\end{figure}
We denote this discrete version of $C_{12}$ by
\begin{equation}\label{DiscreteC12}
\begin{aligned}
\mathcal{C}_{12}[\bar{f}_{k_1}]:= &~~\sum_{k_2+k_3=k_1}\mathcal{K}^{12}_{k_1,k_2,k_3}[(\bar{f}_{k_1}+1)\bar{f}_{k_2}\bar{f}_{k_3}-\bar{f}_{k_1}(\bar{f}_{k_2}+1)(\bar{f}_{k_3}+1)]\\
&~~-2\sum_{k_1+k_3=k_2}\mathcal{K}^{12}_{k_2,k_1,k_3}[(\bar{f}_{k_2}+1)\bar{f}_{k_1}\bar{f}_{k_3}-\bar{f}_{k_2}(\bar{f}_{k_1}+1)(\bar{f}_{k_3}+1)].
\end{aligned}
\end{equation} 

\subsection{The chemical reaction network associated to $C_{12}$}\label{Sec:ChemC12}
For $x\in\mathbb{R}_{>0}^n$ and $\alpha\in \mathbb{R}_{\geq 0}^n$, we denote by $x^\alpha$ the monomial $\Pi_{i=1}^n x_i^{\alpha_i}$.
\begin{definition} Consider a chemical reaction of the form
$$\ce{\alpha_1 X_1  +  \alpha_2 X_2  +  ...    +  \alpha_n X_n ->[\mathcal{K}] \beta_1 X_1  +  \beta_2 X_2  +  ...  +  \beta_n X_n},$$ 
where $\mathcal{K}$ is a positive parameter, called  reaction rate constant. Then the mass-action dynamical system generated by this reaction is
\begin{equation}\label{ExampleEquation}
\dot{x}= \mathcal{K} x^\alpha (\beta-\alpha),
\end{equation}
where $\alpha=(\alpha_1,\cdots,\alpha_n)^T$, $\beta=(\beta_1,\cdots,\beta_n)^T$, $\alpha_i,\beta_i\ge 0$ and  $x=(x_1,\cdots,x_n)^T$, in which $x_i$ is the concentration of the chemical species $X_i$. For the case of a network that contains several reactions
$$\ce{\alpha_1^j X_1^j  +  \alpha_2^j X_2^j  +  ...   +   \alpha_n^j X_n^j ->[\mathcal{K}_j] \beta_1 X_1^j  +  \beta_2^j X_2^j  +  ...   +  \beta_n^j X_n^j},$$ 
for $1\le j\le m$, its associated mass-action dynamical system is given by
\begin{equation}\label{DynSys}
\dot{x}=\sum_{j=1}^m \mathcal{K}_jx^{\alpha^j}(\beta^j-\alpha^j).
\end{equation}
\end{definition}

In this section, we will show  that the system \eqref{DiscreteQuantum1D} has the  form \eqref{DynSys} for a well-chosen set of reactions. 

 If $y\to y'$ and $y'\to y$ are reactions, we combine them together into a {\it ``reversible'' reaction $y\leftrightarrow y'$.}

  We will derive  the system \eqref{DiscreteQuantum1D} from the network of chemical reactions of the form:
\begin{eqnarray}\label{BioChemEq01}
&&\ce{X_{k_2} + X_{k_3} <-> X_{k_1}}\\\label{BioChemEq02}
&&\ce{X_{k_2} + X_{k_1} -> 2X_{k_2} + X_{k_3}},
\end{eqnarray}
for all $k_1,k_2,k_3$ in $\mathbb{I}$ such that $k_2+k_3=k_1$. If we denote by $F_k$ the concentration of the species $X_k$, we will show that, {\it for appropriate choices of the reaction rate constants} in \eqref{BioChemEq01} and \eqref{BioChemEq02}, the differential equations satisfied by $F_k$ according the mass-action kinetics are exactly the same as \eqref{DiscreteQuantum1D}.\\

In order to describe the connection between the mass-action system given by reactions of the form \eqref{BioChemEq01}-\eqref{BioChemEq02} and our system \eqref{DiscreteQuantum1D}, we need to consider several cases.

\bigskip

{\it Case 1:} For $k_2+k_3=k_1$, $k_2\ne k_3$, $k_1,k_2,k_3\in\mathbb{I}$, we consider 
\begin{eqnarray}\label{BioChemEq1a}
&&\ce{X_{k_2} + X_{k_3} <->[2\mathcal{K}^{12}_{k_1,k_2,k_3}] X_{k_1}}\\\label{BioChemEq1b}
&&\ce{X_{k_2} + X_{k_1} ->[2\mathcal{K}^{12}_{k_1,k_2,k_3}] 2X_{k_2} + X_{k_3}},
\end{eqnarray}
and for the reversible reaction \eqref{BioChemEq1a} the forward and backward rate constants are the same, i.e., we choose the reaction rate constants of the three reactions $X_{k_2}+X_{k_3}\to X_{k_1}$, $X_{k_1}\to X_{k_2}+X_{k_3}$, $X_{k_2}+X_{k_1}\to 2X_{k_2}+X_{k_3}$ to be $2\mathcal{K}^{12}_{k_1,k_2,k_3}$.

For example,  consider the reversible reaction \eqref{BioChemEq1a}:
in this reaction, $X_{k_1}$ is created from $X_{k_2}+X_{k_3}$ with the rate $2\mathcal{K}^{12}_{k_1,k_2,k_3}F_{k_2}F_{k_3}$ and $X_{k_1}$ is decomposed into  $X_{k_2}+X_{k_3}$ with the rate $-2\mathcal{K}^{12}_{k_1,k_2,k_3}F_{k_1}$. Therefore, the rate of change of the species $X_{k_1}$ due to this reaction is $ 2\mathcal{K}^{12}_{k_1,k_2,k_3}[F_{k_2}F_{k_3} - F_{k_1}]$.

For the irreversible reaction  \eqref{BioChemEq1b},
 $X_{k_1}$ is lost with the rate $-2\mathcal{K}^{12}_{k_1,k_2,k_3}F_{k_2}F_{k_1}$ to create $2X_{k_2}+X_{k_3}$. Therefore the rate of change of the species $X_{k_1}$ due to this reaction is $ -2\mathcal{K}^{12}_{k_1,k_2,k_3}F_{k_2}F_{k_1}$. By exchanging the roles of $X_{k_2}$ and $X_{k_3}$ in \eqref{BioChemEq1b}, we obtain the rate $ -2\mathcal{K}^{12}_{k_1,k_2,k_3}[F_{k_2}F_{k_1}+F_{k_3}F_{k_1}]$.

Therefore, the total rate of change of $X_{k_1}$ due to the reactions in \eqref{BioChemEq1a}-\eqref{BioChemEq1b} is
\begin{equation}\label{FirstChangeRate}
\begin{aligned}
&2\mathcal{K}^{12}_{k_1,k_2,k_3}[F_{k_2}F_{k_3}- F_{k_1}-F_{k_2}F_{k_1}-F_{k_3}F_{k_1}].
\end{aligned}
\end{equation}

\bigskip

{\it Case 2:} For $2k_2=k_1$, $k_1,k_2\in\mathbb{I}$, we consider
\begin{eqnarray}\label{BioChemEq1a2}
&&\ce{2X_{k_2} <->[\mathcal{K}^{12}_{k_1,k_2,k_3}] X_{k_1}}\\\label{BioChemEq1b2}
&&\ce{X_{k_2} + X_{k_1} ->[2\mathcal{K}^{12}_{k_1,k_2,k_3}] 3X_{k_2} }.
\end{eqnarray}

We choose the reaction rate constant of $2X_{k_2}\to X_{k_1}$ and the reaction rate constant of $X_{k_1}\to 2 X_{k_2}$ to be $\mathcal{K}^{12}_{k_1,k_2,k_3}$. Also, we choose the reaction rate constant of $X_{k_2}+X_{k_1}\to 3X_{k_2}$ to be $2\mathcal{K}^{12}_{k_1,k_2,k_3}$.

Consider the first reaction \eqref{BioChemEq1a2}:
In this reaction, $X_{k_1}$ is created from $2X_{k_2}$ with the rate $\mathcal{K}_{k_1,k_2,k_2}F_{k_2}^2$ and $X_{k_1}$ is decomposed into  $2X_{k_2}$ with the rate $-\mathcal{K}^{12}_{k_1,k_2,k_2}F_{k_1}$. The rate of change of the species $X_{k_1}$ is $\mathcal{K}^{12}_{k_1,k_2,k_2}[F_{k_2}^2 -F_{k_1}]$.

For the second reaction  \eqref{BioChemEq1b2}:
 $X_{k_1}$ is lost with the rate $-2\mathcal{K}^{12}_{k_1,k_2,k_2}F_{k_2}F_{k_1}$ to create $3X_{k_2}$. 

As a result, the rate of change of $X_{k_1}$ due to the reactions  \eqref{BioChemEq1a2}-\eqref{BioChemEq1b2} is
\begin{equation}\label{FirstChangeRate2}
\begin{aligned}
&\mathcal{K}^{12}_{k_1,k_2,k_3}[F_{k_2}^2- F_{k_1}-2F_{k_2}F_{k_1}].
\end{aligned}
\end{equation}

\bigskip

{\it Case 3:} Next, for $k_2=k_3+k_1$, $k_1\ne k_3$, $k_1,k_2,k_3\in\mathbb{I}$, let us look at the rate of change of $X_{k_1}$ in 
\begin{eqnarray}\label{BioChemEq2a}
&&\ce{X_{k_1} + X_{k_3} <->[2\mathcal{K}^{12}_{k_2,k_1,k_3}] X_{k_2}}\\\label{BioChemEq2b}
&&\ce{X_{k_2} + X_{k_1} ->[2\mathcal{K}^{12}_{k_2,k_1,k_3}] 2X_{k_1} + X_{k_3}}\\\label{BioChemEq2c}
&&\ce{X_{k_2} + X_{k_3} ->[2\mathcal{K}^{12}_{k_2,k_1,k_3}] X_{k_1} + 2X_{k_3}},
\end{eqnarray}

For \eqref{BioChemEq2a}, the rate of change of $X_{k_1}$ is $2\mathcal{K}^{12}_{k_2,k_1,k_3}[F_{k_2}-F_{k_1}F_{k_3}]$. For \eqref{BioChemEq2b}, the rate of change of $X_{k_1}$ is $2\mathcal{K}^{12}_{k_2,k_1,k_3}F_{k_1}F_{k_2}.$ By exchanging the roles of $X_1$ and $X_3$, we obtain the rate $2\mathcal{K}^{12}(k_2,k_1,k_3)[F_{k_1}F_{k_2}+F_{k_2}F_{k_3}]$.\\
Therefore, the rate of change of $X_{k_1}$ due to reactions in \eqref{BioChemEq2a}-\eqref{BioChemEq2c} is
\begin{equation}\label{SecondChangeRate}
\begin{aligned}
&-2\mathcal{K}^{12}_{k_2,k_1,k_3}[F_{k_1}F_{k_3}- F_{k_2}-F_{k_2}F_{k_3}-F_{k_1}F_{k_2}].
\end{aligned}
\end{equation} 

\bigskip

{\it Case 4:} Now, for $k_2=2k_1$, $k_1,k_2\in\mathbb{I}$, let us look at the rate of change of $X_{k_1}$ in 
\begin{eqnarray}\label{BioChemEq2a2}
&&\ce{2X_{k_1} <->[\mathcal{K}^{12}_{k_2,k_1,k_1}] X_{k_2}}\\\label{BioChemEq2b2}
&&\ce{X_{k_2} + X_{k_1} ->[2\mathcal{K}^{12}_{k_2,k_1,k_3}] 3X_{k_1}},
\end{eqnarray}
For \eqref{BioChemEq2a2}, the rate of change of $X_{k_1}$ is $2\mathcal{K}^{12}_{k_2,k_1,k_3}[F_{k_2}-F_{k_1}^2]$.
For \eqref{BioChemEq2b2}, the rate of change of $X_{k_1}$ is $4\mathcal{K}^{12}_{k_2,k_1,k_3}F_{k_1}F_{k_2}.$ 
Therefore, the rate of change of $X_{k_1}$ due to the reactions \eqref{BioChemEq2a2}-\eqref{BioChemEq2b2} is
\begin{equation}\label{SecondChangeRate2}
\begin{aligned}
&-2\mathcal{K}^{12}_{k_2,k_1,k_3}[F_{k_1}^2- F_{k_2}-2F_{k_1}F_{k_2}].
\end{aligned}
\end{equation}

From \eqref{FirstChangeRate}, \eqref{FirstChangeRate2}, \eqref{SecondChangeRate}, \eqref{SecondChangeRate2},
the total  rate of change of $X_{k_1}$ is
\begin{equation}\label{TotalChangeRate}
\begin{aligned}
&~~\sum_{k_2+k_3=k_1, k_2< k_3}2\mathcal{K}^{12}_{k_1,k_2,k_3}[(F_{k_1}+1)F_{k_2}F_{k_3}-F_{k_1}(F_{k_2}+1)(F_{k_3}+1)]\\
&~~+\sum_{2k_2=k_1}\mathcal{K}^{12}_{k_1,k_2,k_2}[(F_{k_1}+1)F_{k_2}F_{k_2}-F_{k_1}(F_{k_2}+1)(F_{k_2}+1)]\\
&~~-\sum_{k_1+k_3=k_2}2\mathcal{K}^{12}_{k_2,k_1,k_3}[(F_{k_2}+1)F_{k_1}F_{k_3}-F_{k_2}(F_{k_1}+1)(F_{k_3}+1)],
\end{aligned}
\end{equation} 
which can be written as
\begin{equation}\label{TotalChangeRateEquation}
\begin{aligned}
\dot{F}_{k_1}\ =&~~\sum_{k_2+k_3=k_1}\mathcal{K}^{12}_{k_1,k_2,k_3}[(F_{k_1}+1)F_{k_2}F_{k_3}-F_{k_1}(F_{k_2}+1)(F_{k_3}+1)]\\
&~~-2\sum_{k_1+k_3=k_2}\mathcal{K}^{12}_{k_2,k_1,k_3}[(F_{k_2}+1)F_{k_1}F_{k_3}-F_{k_2}(F_{k_1}+1)(F_{k_3}+1)],
\end{aligned}
\end{equation} 
which shows that the system of differential equations satisfied by the concentrations $F_k$ is exactly the same as the system of differential equations \eqref{DiscreteQuantum1D} satisfied by the densities $f_k$.
\subsection{A change of variables}
In this section, we introduce a change of variables that will help us to investigate the dynamics of the system \eqref{TotalChangeRateEquation}.

Define $$G_k=\frac{F_k}{F_k+1},$$
then 
$$F_k=\frac{G_k}{1-G_k},$$
and 
$$F_{k_3}+F_{k_1}F_{k_3}+F_{k_2}F_{k_3}-F_{k_1}F_{k_2}=\frac{G_{k_3}-G_{k_1}G_{k_2}}{(1-G_{k_1})(1-G_{k_2})(1-G_{k_3})},$$
$$F_{k_1}+F_{k_1}F_{k_2}+F_{k_1}F_{k_3}-F_{k_3}F_{k_2}=\frac{G_{k_1}-G_{k_2}G_{k_3}}{(1-G_{k_1})(1-G_{k_2})(1-G_{k_3})}.$$
Notice that $0<F_k<\infty$ and $0<G_k<1$.\\
The system  \eqref{TotalChangeRateEquation} is converted into
\begin{eqnarray}\label{DiscreteQuantumConverted}
\frac{\dot{G}_{k_1}}{(1-G_{k_1})^2}\nonumber
&=& \tilde{\mathcal{C}}_{12}[G](k_1):=2\sum_{k_1+k_2=k_3}\mathcal{K}^{12}_{k_1,k_2,k_3}\frac{G_{k_3}-G_{k_1}G_{k_2}}{(1-G_{k_1})(1-G_{k_2})(1-G_{k_3})}\\
&+&\sum_{k_1=k_2+k_3}\mathcal{K}^{12}_{k_1,k_2,k_3}\frac{-G_{k_1}+G_{k_2}G_{k_3}}{(1-G_{k_1})(1-G_{k_2})(1-G_{k_3})}, \forall k_1\in\mathbb{I}.
\end{eqnarray}
Suppose that  $G$ represents the column vector $(G_1,\dots,G_I)^T$. Let us also denote by  $\bar{X}_k$, the vector 
$$\left( {\begin{array}{cc}
 0\\       \cdots\\ 1\\ \cdots \\   0  \end{array} } \right),$$
in which the only element that different from $0$ is the $k$-th one.

Also, for  $k_1\ne k_2$, we denote
$$K_{\bar{X}_{k_1}+\bar{X}_{k_2} \rightarrow \bar{X}_{k_3}}(G):=2\mathcal{K}^{12}_{k_1,k_2,k_3}\frac{G_{k_1}G_{k_2}}{(1-G_{k_1})(1-G_{k_2})(1-G_{k_3})},$$
$$K_{\bar{X}_{k_3} \rightarrow X_{k_1}+\bar{X}_{k_2}}(G):=2\mathcal{K}^{12}_{k_1,k_2,k_3}\frac{G_{k_3}}{(1-G_{k_1})(1-G_{k_2})(1-G_{k_3})},$$
$$\mathcal{K}_{\bar{X}_{k_1}+\bar{X}_{k_2}\leftrightarrow \bar{X}_{k_3}}:=2\mathcal{K}^{12}_{k_1,k_2,k_3}.$$

 Otherwise, if $k_1= k_2$ , we denote
$$K_{2\bar{X}_{k_1}\rightarrow X_{k_3}}(G):=\mathcal{K}^{12}_{k_1,k_1,k_3}\frac{G_{k_1}G_{k_2}}{(1-G_{k_1})^2(1-G_{k_3})},$$
$$K_{\bar{X}_{k_3} \rightarrow 2\bar{X}_{k_1}}(G):=\mathcal{K}^{12}_{k_1,k_1,k_3}\frac{G_{k_3}}{(1-G_{k_1})^2(1-G_{k_3})},$$
$$\mathcal{K}_{2\bar{X}_{k_1}\leftrightarrow \bar{X}_{k_3}}:=2\mathcal{K}^{12}_{k_1,k_1,k_3}.$$

Using these notations,  the   system \eqref{DiscreteQuantumConverted} could be rewritten as:
\begin{eqnarray}\label{EqG}
{\dot{G}}
&=&\mbox{diag}\left( {\begin{array}{cc}
  {(1-G_1)^2}\\       \cdots \\   {(1-G_I)^2}  \end{array} } \right)\times\\\nonumber
& &\times\sum_{k_1+k_2=k_3}\left[K_{\bar{X}_{k_1}+\bar{X}_{k_2} \rightarrow \bar{X}_{k_3}}(G)-K_{\bar{X}_{k_3} \rightarrow \bar{X}_{k_1}+\bar{X}_{k_2} }(G)\right](\bar{X}_{k_3}-\bar{X}_{k_1}-\bar{X}_{k_2} ).
\end{eqnarray}

Equivalently, we can also write
\begin{equation}\label{EqG2}
{\dot{G}}
=\mbox{diag}\left( {\begin{array}{cc}
  {(1-G_1)^2}\\       \cdots \\   {(1-G_I)^2}  \end{array} } \right)\sum_{y\leftrightarrow y'}\left[K_{y \rightarrow y'}(G)-K_{y' \rightarrow y}(G)\right](y'-y ),
\end{equation}
where $y\leftrightarrow y'$ belongs to the set of reversible reactions
\begin{eqnarray}\label{ChemicalNetworkReversible1a}
\bar{X}_{k_1}+\bar{X}_{k_2}&\longleftrightarrow &\bar{X}_{k_3},
\end{eqnarray}
with $k_1+k_2=k_3$.
\subsection{Convergence to equilibrium}

\begin{theorem}\label{TheoremQ12} For any positive initial condition, the solution $$f(t)=(f_p(t))_{p\in\mathcal{L}_R}$$ of the discrete quantum Boltzmann equation \eqref{DiscreteQuantum} converges to an equilibrium state $f^*=(f_p^*)_{p\in\mathcal{L}_R}$. For each ray $\{kP_0\}_{k\geq 1}$ there exists a positive constant $\rho({P_0})$ such that if $p=kP_0$ then $$f^*_p=\frac{1}{e^{k\rho({P_0})}-1}.$$ Moreover, the solution $f(t)$ of \eqref{DiscreteQuantum} converges to $f^*$ exponentially fast in the following sense: there exist positive constants $C_1$, $C_2$ such that $$\max_{p\in\mathcal{L}_R}|f_p(t)-f_p^*|<C_1e^{-C_2t}.$$
\end{theorem}
\begin{proof}
By using the decoupling and the change of variables discussed in the previous sections, for each ray $\{kP_0\}_{k\geq 1}$, we can reduce the study of $f$ to $F$, which satisfies \eqref{TotalChangeRateEquation}. From $F$, we can switch to study $G$, which is the solution of \eqref{EqG2}.
\bigskip
{\\\it Step 1: The Lyapunov function.}
 We recall that  \eqref{DiscreteQuantumConverted} could be rewritten under the form
\begin{equation}\label{EQ}
{\dot{G}}=\mbox{diag}\left( {\begin{array}{cc}
  {(1-G_1)^2}\\       \cdots \\   {(1-G_I)^2}  \end{array} } \right)\sum_{y \leftrightarrow y'}\left[K_{y \to y'}(G)-K_{y'\to y}(G)\right](y'-y).
\end{equation}
We define the  function
\begin{equation}\label{Lyapunov}
L(G)=\sum_{k=1}^I \left(\log (1-G_k)+\frac{G_k \log G_k}{1-G_k}-\frac{\log G_k^*}{1-G_k}\right),
\end{equation}
where $G^*_k=\frac{1}{e^{k\rho}}$, for some $\rho>0$,
and we will show that $L$ is a Lyapunov function for the system \eqref{DiscreteQuantumConverted}.

We have
\begin{equation}
\nabla L=\left( {\begin{array}{cc}
  \frac{1}{(1-G_1)^2}\log \frac{G_1}{G^*_1}\\       \cdots \\   \frac{1}{(1-G_I)^2}\log \frac{G_I}{G^*_I}   \end{array} } \right),
\end{equation}
which implies that 
\begin{eqnarray}\label{eq1}
\mbox{diag}\left( {\begin{array}{cc}
  {(1-G_1)^2}\\       \cdots \\   {(1-G_I)^2}  \end{array} } \right)\cdot(y'-y)\cdot\nabla L&=&\log \left(\frac{G}{G^*}\right)^{y'-y}\\\nonumber
& = &\log\left(\frac{G}{G^*}\right)^{y'}-\log\left(\frac{G}{G^*}\right)^{y}.
\end{eqnarray}

If we define
$$\mathcal{H}_{y,y'}(G)=\frac{K_{y\to y'}(G)}{\mathcal{K}_{y\leftrightarrow y'}G^y},$$
then $\mathcal{H}_{y,y'}=\mathcal{H}_{y',y}$ for $y$ and $y'$ as in \eqref{ChemicalNetworkReversible1a}. Moreover,
we have
\begin{equation}\label{eq2}
\begin{aligned}
& K_{y\to y'}(G)-K_{y'\to y}(G)=\\
=& \mathcal{K}_{y\leftrightarrow y'}G^y \mathcal{H}_{y,y'}(G)-\mathcal{K}_{y\leftrightarrow y'}G^{y'} \mathcal{H}_{y,y'}(G)\\
=&\mathcal{K}_{y\leftrightarrow y'}\mathcal{H}_{y,y'}(G)[G^y-G^{y'}]\\
=&\mathcal{K}_{y\leftrightarrow y'}(G^*)^y \mathcal{H}_{y,y'}(G)\left[\frac{G^y }{(G^*)^y}
-\frac{G^{y'} }{(G^*)^{y'} }\right],
\end{aligned}
\end{equation}
since $(G^*)^y=(G^*)^{y'}$.\\
Combining \eqref{EQ}, \eqref{eq1} and \eqref{eq2}, we obtain
\begin{equation}
\begin{aligned}
& \ \dot{G}\cdot\nabla L = \\
 =  & \sum_{y\leftrightarrow y'}\left[\log\left(\frac{G}{G^*}\right)^{y'}-\log\left(\frac{G}{G^*}\right)^{y}\right] \mathcal{K}_{y\leftrightarrow y'}(G^*)^y \mathcal{H}_{y,y'}(G)\left[\frac{G^y }{(G^*)^y}
-\frac{G^{y'} }{(G^*)^{y'} }\right]\\
 \leq  & \  0,
\end{aligned}
\end{equation}
since $\mathrm{log}$ is an increasing function. 
Also, note that the above inequality is strict unless
\begin{equation}\label{C12equilibrium}
\frac{G^y }{(G^*)^y}
=\frac{G^{y'} }{(G^*)^{y'} },\end{equation}
for all reactions $y \leftrightarrow y'.$
\\Since $(G^*)^y=(G^*)^{y'}$ for all reactions $y \leftrightarrow y',$ this implies $G^{*}_{k_1}\cdot G^{*}_{k_2} = G^*_{k_1+k_2}$ for all $k_1$ and $k_2$ such that $k_1+k_2\le I$. As a consequence   $G^*_k=e^{-\rho k}$, for some positive constant $\rho$. Moreover, \eqref{C12equilibrium}  implies that at equilibrium  $(G)^y=(G)^{y'}$  for all reactions $y \leftrightarrow y',$ which leads to $G_k=e^{-\rho' k}$, for some positive constant $\rho'$. 
\\ By the conservation relation
$$\sum_{k=1}^Ik\frac{G_k}{1-G_k}=\sum_{k=1}^Ik\frac{G_k^*}{1-G_k^*},$$
we deduce that 
$$\sum_{k=1}^Ik\frac{e^{-\rho k}}{1-e^{-\rho k}}=\sum_{k=1}^Ik\frac{e^{-\rho' k}}{1-e^{-\rho' k}}.$$
By the monotonicity of the function $\rho\to\frac{e^{-\rho k}}{1-e^{-\rho k}}$, we conclude that $\rho=\rho'$, i.e., $G^*$ is the only equilibrium point that satisfies the same conservation relation as the initial condition.

Now, we will prove that there exists exactly one critical point of the Lyapunov function $L$ within each invariant set $$\mathfrak{S}_{c}:=\left\{\sum_{k=1}^Ik\frac{G_k}{1-G_k}=c\right\}.$$
Since 
$$\nabla L=\mbox{diag}\left( {\begin{array}{cc}
  \frac{1}{(1-G_1)^2}\\       \cdots \\   \frac{1}{(1-G_I)^2} \end{array} } \right)[\log G-\log G^*],$$
the projection of $\nabla L$ on the tangent space to the set $\mathfrak{S}_{c}$ is $0$ if and only if there exists a constant $\varrho$ such that
$$\nabla L=\varrho \cdot \nabla \left(\sum_{k=1}^Ik\frac{G_k}{1-G_k}\right),$$
which is equivalent with $$\mbox{diag}\left( {\begin{array}{cc}
  \frac{1}{(1-G_1)^2}\\       \cdots \\   \frac{1}{(1-G_I)^2} \end{array} } \right)[\log G-\log G^*]=\varrho\left( {\begin{array}{cc}
  \frac{1}{(1-G_1)^2}\\       \cdots \\   \frac{I}{(1-G_I)^2} \end{array} } \right).$$
A direct consequence of the above is the following system of identities
\begin{eqnarray*}
\log G_1 - \log G_1^* &=& \varrho,\\
\log G_2 - \log G_2^* &=& 2\varrho,\\
&\cdots&\\
\log G_I - \log G_I^* &=& I\varrho,
\end{eqnarray*}
yielding
$$\frac{G_k}{G_k^*}=e^{k\varrho}, \ \ \ \forall k\in \{1,\cdots,I\}.$$
Moreover, since $G_k$ and $G_k^*$ satisfy the same conservation law then it follows that $G=G^*$. This implies that $G^*$ is the only critical point of $L$ on the invariant set $\mathfrak{S}_c$.
\bigskip
{\\\it Step 2: Differential inclusions and persistence.}
 Now let us observe that \eqref{DiscreteQuantum1D} could be regarded as a $\mathcal{K}$-variable mass-action system for the reversible network \eqref{ChemicalNetworkReversible1a}. For this  we write
$$F_{k''}+F_kF_{k''}+F_{k'}F_{k''}=(1+F_k+F_{k'})F_{k''},$$
and note that $1+F_k+F_{k'}$ is bounded below by $1$ and above by $1+2C$, where
$$C=\sum_{k=1}^I kF_k.$$

Therefore, the  results of \cite{Craciun:2015:TDI} about persistence of $\mathcal{K}$-variable reversible mass-action systems can be applied and we conclude that the system is persistent. Alternatively, we  can also use the Petri net argument of \cite{AngeliLeenheerSontag:2011:PRF}, to prove that the system is persistent, as follows. Note that $F_k$ is the density function of the species $X_k$. It is straightforward that each {\it siphon} is $\{X_1, X_2,\cdots, X_I\}$, which contains the support of the {\it $P$-semiflow} (see \cite{AngeliLeenheerSontag:2011:PRF} for the definition of siphons and P-semiflows) given by 
$$\sum_{k=1}^IkF_k=\mbox{constant}.$$
As a result, the Petri net theory developed in \cite{AngeliLeenheerSontag:2011:PRF} can be applied and it follows that the system is persistent.

Therefore, by using the existence of the globally defined strict Lyapunov function $L$, and the LaSalle invariance principle, it follows that all trajectories converge to the unique positive equilibrium $G^*$ that we discussed in~Step~1.
 \bigskip
{\\\\\it Step 3: Exponential rate of convergence.}
 Define 
\begin{eqnarray}\label{DefineR}\nonumber
&&\mathcal{R}(G)=\\
&=&\mbox{diag}\left( {\begin{array}{cc}
  {(1-G_1)^2}\\       \cdots \\   {(1-G_I)^2}  \end{array} } \right)\sum_{y \leftrightarrow y'}\left[K_{y \to y'}(G)-K_{y'\to y}(G)\right](y'-y)\\\nonumber
&=&\mbox{diag}\left( {\begin{array}{cc}
  {(1-G_1)^2}\\       \cdots \\   {(1-G_I)^2}  \end{array} } \right)\sum_{y \leftrightarrow y'}[\mathcal{K}_{y\leftrightarrow y'}G^y -\mathcal{K}_{y\leftrightarrow y'}G^{y'}] \mathcal{H}_{y,y'}(G)(y'-y),
\end{eqnarray}
and define
\begin{eqnarray}\label{DefineRR}\nonumber
\mathcal{S}(G)
&=&\sum_{y \leftrightarrow y'}[\mathcal{K}_{y\leftrightarrow y'}G^y -\mathcal{K}_{y\leftrightarrow y'}G^{y'}]\mathcal{H}_{y,y'}(G)(y'-y).
\end{eqnarray}
\\ Following \cite{CraciunFeinberg:2005:MEI}, we compute the Jacobian of $\mathcal{S}$ at the equilibrium point $G^*$, applied to an arbitrary vector $\delta\ne 0$  that belongs to the span of the vectors $y'-y$
\begin{eqnarray}\label{JacR}
\mbox{Jac}(\mathcal{S}(G^*))\delta =\sum_{y \leftrightarrow y'} \mathcal{K}_{y\leftrightarrow y'}(G^*)^y ((y-y')*\delta)\mathcal{H}_{y,y'}(G^*)(y-y'),
\end{eqnarray}
in which the inner product $*$ is defined as
$$y*\delta=\sum_{1}^I \frac{y_k \delta_k}{G_k}.$$
Therefore 
\begin{eqnarray}\label{JacR1}
&& [\mbox{Jac}(\mathcal{S}(G^*))\delta]*\delta =\\\nonumber
& =& \sum_{y \leftrightarrow y'} \mathcal{K}_{y\leftrightarrow y'}(G^*)^y\mathcal{H}_{y,y'}(G^*)[(y-y')*\delta][(y'-y)*\delta]<0.
\end{eqnarray}
Now, we compute the Jacobian of $\mathcal{R}$ at the equilibrium point $G^*$,
\begin{eqnarray}\nonumber
& &\mbox{Jac}(\mathcal{R}(G^*))\\\nonumber
&=&\mbox{diag}\begin{bmatrix}
\partial_{G_1} {(1-G^*_1)^2}\mathcal{S}(G^*)_1\\\cdots     \\    \partial_{G_I} {(1-G^*_I)^2}\mathcal{S}(G^*)_I\end{bmatrix}+\mbox{diag}\begin{bmatrix}
  {(1-G^*_1)^2}\\\cdots     \\    {(1-G^*_I)^2}\end{bmatrix} \mbox{Jac}(\mathcal{S}(G^*))\\\nonumber
&=&\mbox{diag}\begin{bmatrix}
 {(1-G^*_1)^2} \\\cdots     \\   {(1-G^*_I)^2} \end{bmatrix} \mbox{Jac}(\mathcal{S}(G^*)),
\end{eqnarray}
where the second equality is due to the fact that since $G^*$ is an equilibrium we have that $\mathcal{S}(G^*)=0$.\\
Since $$D:=\mbox{diag}\begin{bmatrix}
  {(1-G^*_1)^2} \\\cdots     \\    {(1-G^*_I)^2}\end{bmatrix} $$
is a diagonal matrix and $ A:=\mbox{Jac}(\mathcal{S}(G^*))$ is negative definite, then $D^{1/2}AD^{1/2}$ is also negative definite with respect to this inner product. Since 
$$\mbox{det}(DA-\lambda Id)=\mbox{det}(D^{1/2}AD^{1/2}-\lambda Id),~~\forall \lambda\in\mathbb{R},$$
it follows that $D^{1/2}AD^{1/2}$ and $DA$ have the same eigenvectors, so $DA$ is  negative definite. In other words, $\mbox{Jac}(\mathcal{R}(G^*))$ is negative definite.
The exponential rate of convergence
$$\max\{|G_1(t)-G_1^*|,\cdots, |G_I(t)-G_I^*|\}\leq C_1 e^{-C_2 t}.
$$
 then follows from the fact that the Jacobian above is negative definite. This  leads to the conclusion of the theorem.

\end{proof}
\begin{remark}
The Lyapunov function \eqref{Lyapunov} in the variable $F$ reads
\begin{equation}\label{LyapunovF}
L(F)=\sum_{k=1}^I [F_k\log F_k-(1+F_k)\log (1+F_k)+(\log (F_{k}^*+1)-\log F_k^*)(F_{k}+1)],
\end{equation} 
and it is  a strictly convex function.
\end{remark}
\begin{remark}
If the intersection between the ray $\{kP_0\}_{k\geq 1}$ and $\mathcal{L}_R$ contains a single point, then the solution $f(t)$ of \eqref{DiscreteQuantum} has $f_{P_0}\equiv 0$, so $f_{P_0}\equiv$ constant. 
\end{remark}
\section{A reaction network approach for the case of $C_{13}$ and $C_{22}$}\label{Sec:Qnm}
\subsection{The dynamical system associated to $C_{13}$}
As we discussed in the Introduction, we are also interested in the dynamics given by the discrete model of the collision operator $C_{13}$, described in \eqref{QuantumBoltzmannLinda13}. 
\\\\ Let $\mathcal{L}_R$ denote the lattice of integer points 
$$\mathcal{L}_R=\{p\in\mathbb{Z}^3~~~~|~~~~|p|<R\}.$$
The discretized quantum Boltzmann equation for $C_{13}$ reads 
\begin{equation}\label{DiscreteQuantuma}\begin{aligned}
&\dot{f}_{p_1}=C_{13}^D[f_{p_1}]:=\\
:=
&\quad\sum_{\substack{p_2,p_3,p_4\in\mathcal{L}_R,\\ p_1=p_2+p_3+p_4,\\ \mathcal{E}(p_1)=\mathcal{E}(p_{2})+\mathcal{E}(p_{3})+\mathcal{E}(p_{4})}}{K}^{13}_{p_1,p_2,p_3,p_4}\{(f_{p_1}+1)f_{p_2}f_{p_3}f_{p_4}-(f_{p_2}+1)(f_{p_3}+1)(f_{p_4}+1)f_{p_1}\}\\
&\quad -3\sum_{\substack{p_2,p_3,p_4\in\mathcal{L}_R,\\ p_2=p_1+p_3+p_4,\\ \mathcal{E}(p_2)=\mathcal{E}(p_{1})+\mathcal{E}(p_{3})+\mathcal{E}(p_{4})}}{K}^{13}_{p_2,p_1,p_3,p_4}\left\{(f_{p_2}+1)f_{p_1}f_{p_3}f_{p_4}-(f_{p_1}+1)(f_{p_3}+1)(f_{p_4}+1)f_{p_2}\right\},~~
\end{aligned}\end{equation}
for all $p_1$ in $\mathcal{L}_R$, where $\mathcal{E}(p)$ is defined in \eqref{E3a}.

Similar  to the $C_{12}$ case, when $p=0$, ${K}^{13}_{p_1,p_2,p_3,p_4}=0$, and we obtain 
\begin{eqnarray*}\label{DiscreteQuantumIndex0}
\dot{f}_0=0,
\end{eqnarray*}
which means  $f_0(t)$ is a constant for all time $t$, and we can assume $f_0(t)=0$ for all $t$.
\\ Since in the first sum of \eqref{DiscreteQuantuma}, we consider $(p_1,p_2,p_3,p_4)$ satisfying
\begin{equation}\label{EnergyC13}
p_1 \ = \ p_2 \ + \ p_3\  +\ p_4, \ \ \ \  \mathcal{E}(p_1)\ = \ \mathcal{E}(p_2)\ +\ \mathcal{E}(p_3) \ + \mathcal{E}(p_4),
\end{equation}
we infer that there exists a vector $P$  and  $k_1$, $k_2$, $k_3, k_4\geq0$, $k_1, k_2, k_3, k_4\in\mathbb{Z}$ such that 
$$p_1=k_1 P;~~~p_2=k_2 P;~~~p_3=k_3 P;~~~p_4=k_4 P;~~~k_1=k_2+k_3+k_4.$$  Using the same arguments as the case of $C_{12}$, we can deduce that Equation \eqref{DiscreteQuantuma} for $C_{13}$ is equivalent with the following family of decoupled systems for $k_1\in\mathbb{I}=\{1,2,\dots,I\}$ where $P$ is the closest point to the origin among the lattice points on its ray:
\begin{equation}\label{DiscreteC13Pre}\begin{aligned}
&\dot{f}_{k_1P}=\\
=&\quad\sum_{\substack{k_2,k_3,k_4\in
\mathbb{I},\\ k_1=k_2+k_3+k_4}}{K}^{13}_{k_1P,k_2P,k_3P,k_4P}\{(f_{k_1P}+1)f_{k_2P}f_{k_3P}f_{k_4P}\\
& -f_{k_1P}(f_{k_2P}+1)(f_{k_3P}+1)(f_{k_4P}+1)\}\\
&\quad-{3}\sum_{\substack{k_2,k_3,k_4\in
\mathbb{I},\\ k_2=k_1+k_3+k_4}}{K}^{13}_{k_2P,k_1P,k_3P,k_4P}\{(f_{k_2P}+1)f_{k_1P}f_{k_3P}f_{k_4P}\\
& -f_{k_2P}(f_{k_1P}+1)(f_{k_3P}+1)(f_{k_4P}+1)\}.\end{aligned}
\end{equation}
Denoting $f_{kP}$ by $F_k$ (with $k\in \mathbb{I}$) and ${K}^{12}_{k_1P,k_2P,k_3P,k_4P}$  by $\mathcal{K}^{12}_{k_1,k_2,k_3,k_4}$, we obtain 
\begin{equation}\label{DiscreteC13}\begin{aligned}
\dot{F}_{k_1}
=&\ \mathcal{C}_{13}[F](k_1)=\quad\sum_{k_1=k_2+k_3+k_4}\mathcal{K}^{13}_{k_1,k_2,k_3,k_4}\{(F_{k_1}+1)F_{k_2}F_{k_3}F_{k_4}-\\
& -F_{k_1}(F_{k_2}+1)(F_{k_3}+1)(F_{k_4}+1)\}
\\
&\quad-3\sum_{k_1+k_2+k_3=k_4}\mathcal{K}^{13}_{k_1,k_2,k_3,k_4}\{(F_{k_4}+1)F_{k_1}F_{k_2}F_{k_3}-\\
&-F_{k_4}(F_{k_1}+1)(F_{k_2}+1)(F_{k_3}+1)\}, ~\forall k_1\in\mathbb{I}.\end{aligned}
\end{equation}
In order to ensure that all the variables $F_k$ are coupled with each other, let us assume that $I\geq 4$. We  have the following conservation of energy for $C_{13}$
\begin{equation}\label{ConservationE1}
\sum_{k=1}^I k\dot{F_k}=0,
\end{equation}
or equivalently
\begin{equation}\label{ConservationE2}
\sum_{k=1}^I k{F_k}=\mbox{const}.
\end{equation}
Similar  to the case of $C_{12}$, we define $$G_k=\frac{F_k}{F_k+1},$$
and then we have
$$F_k=\frac{G_k}{1-G_k}.$$
Note that, similar to the previous section, $0<F_k<\infty$ and $0<G_k<1$.\\
The system  \eqref{DiscreteC13} can be now written
\begin{equation}\label{DiscreteQuantumaConverteda}\begin{aligned}
&\frac{\dot{G}_{k_1}}{(1-G_{k_1})^2}=\overline{\mathcal{C}_{13}}[G]:=\\
:=&\quad\mathcal{K}^{13}_{k_1,k_2,k_3,k_4}\sum_{\substack{k_1=k_{2}+k_3+k_4,\\
|k_1|=|k_2|+|k_3|+|k_4|
}}\frac{G_{k_2}G_{k_3}G_{k_4}-G_{k_1}}{(1-G_{k_1})(1-G_{k_2})(1-G_{k_3})(1-G_{k_4})}
\\
&-3\mathcal{K}^{13}_{k_2,k_1,k_3,k_4}\sum_{\substack{k_2=k_{1}+k_3+k_4,\\
|k_2|=|k_1|+|k_3|+|k_4|
}}\frac{G_{k_1}G_{k_3}G_{k_4}-G_{k_2}}{(1-G_{k_1})(1-G_{k_2})(1-G_{k_3})(1-G_{k_4})}.\end{aligned}
\end{equation}
This system can also be rewritten as
\begin{equation}\begin{aligned}
{\dot{G}}
=&\mbox{diag}\left( {\begin{array}{cc}
  {(1-G_1)^2}\\       \cdots \\   {(1-G_I)^2}  \end{array} } \right)\times\\
 &\times\sum_{\substack{k_1=k_{2}+k_3+k_4,\\
|k_1|=|k_2|+|k_3|+|k_4|
}}\big[K_{\bar{X}_{k_{2}}+\bar{X}_{k_{3}}+ \bar{X}_{k_{4}}\rightarrow \bar{X}_{k_{1}}}(G)-K_{\bar{X}_{k_{1}}\rightarrow \bar{X}_{k_{2}}+\bar{X}_{k_{3}}+\bar{X}_{k_{4}}}(G)\big](\bar{X}_{k_{1}}-\bar{X}_{k_{2}}-\bar{X}_{k_{3}}-\bar{X}_{k_{4}}).\end{aligned}
\end{equation}
where $\bar{X}_k$ is, as mentioned earlier, the vector 
$$\left( {\begin{array}{cc}
 0\\       \cdots\\ 1\\ \cdots \\   0  \end{array} } \right),$$
in which the only element that is $1$ is the $k$-th one,
and 
$$K_{\bar{X}_{k_2}+\bar{X}_{k_3}+ \bar{X}_{k_4}\rightarrow \bar{X}_{k_{1}} }(G):=\mathcal{K}^{13}_{k_1,k_2,k_3,k_4}\frac{G_{k_1}}{(1-G_{k_1})(1-G_{k_2})(1-G_{k_3})(1-G_{k_4})},$$
$$K_{\bar{X}_{k_{1}}\rightarrow \bar{X}_{k_2}+\bar{X}_{k_3}+ \bar{X}_{k_4}}(G):=\mathcal{K}^{13}_{k_1,k_2,k_3,k_4}\frac{G_{k_2}G_{k_3}G_{k_4}}{(1-G_{k_1})(1-G_{k_2})(1-G_{k_3})(1-G_{k_4})}.$$
We can also write 
\begin{eqnarray}\nonumber
{\dot{G}}
&=&\mbox{diag}\left( {\begin{array}{cc}
  {(1-G_1)^2}\\       \cdots \\   {(1-G_I)^2}  \end{array} } \right)\sum_{y\leftrightarrow y'}\left[K_{y \rightarrow y'}(G)-K_{y' \rightarrow y}(G)\right](y'-y ),
\end{eqnarray}
where $y\leftrightarrow y'$ rang over the reversible reactions shown above.
\begin{theorem}\label{TheoremQnm}  For any initial condition, the solution $$f(t)=(f_p(t))_{p\in\mathcal{L}_R}$$ of the quantum Boltzmann equation \eqref{DiscreteQuantuma} converges to an equilibrium state $f^*=(f_p^*)_{p\in\mathcal{L}_R}$. For each ray $\{kP_0\}_{k\geq 1}$ that intersects $\mathcal{L}_R$ in at least $4$ points there exists a constant $\rho_{P_0}$ such that if $p=kP_0$ then $$f^*_p=\frac{1}{e^{k\rho_{P_0}}-1}.$$ Moreover, the solution $f(t)$ of \eqref{DiscreteQuantuma} converges to $f^*$ exponentially fast in the following sense: there exists positive constants $C_1$, $C_2$ such that $$\max_{p\in\mathcal{L}_R}|f_p(t)-f_p^*|<C_1e^{-C_2t}.$$
\end{theorem}
\begin{proof}
The proof of Theorem \ref{TheoremQnm} then follows exactly from the same Lyapunov function
\eqref{Lyapunov} and arguments as in Theorem \ref{TheoremQ12}.
\end{proof}
\subsection{The dynamical system associated to $C_{22}$}
Let us consider a discretized version of the quantum Boltzmann model associated to the collision operator given by $C_{22}$:
\\\\ Let $\mathcal{L}_R$ denote the lattice of integer points 
$$\mathcal{L}_R=\{p~~|~~|p|\in\mathbb{Z}^3, |p|<R\}.$$
The discretized quantum Boltzmann equation associated to $C_{22}$ reads $\forall p_1\in \mathcal{L}_R$
\begin{equation}\label{DiscreteQuantumb}\begin{aligned}
&\dot{f}_{p_1}=C_{22}^D[f_{p_1}]:=\\
:=
&\quad\sum_{\substack{p_2,p_3,p_4\in\mathcal{L}_R,\\ p_1+p_2=p_3+p_4,\\ \mathcal{E}(p_1)+\mathcal{E}(p_{2})=\mathcal{E}(p_{3})+\mathcal{E}(p_{4})}}{K}^{13}_{p_1,p_2,p_3,p_4}\{(f_{p_1}+1)(f_{p_2}+1)f_{p_3}f_{p_4}-f_{p_1}f_{p_2}(f_{p_3}+1)(f_{p_4}+1)\},~~
\end{aligned}\end{equation}
where  $\mathcal{E}(p)$ is defined in \eqref{E3a}.

Similar  to the $C_{12}$ case, when $p=0$, ${K}^{22}_{p_1,p_2,p_3,p_4}=0$, and we obtain 
\begin{eqnarray*}\label{DiscreteQuantumIndex0}
\dot{f}_0=0,
\end{eqnarray*}
which means  $f_0(t)$ is a constant for all time $t$. As a consequence, we can suppose that $f_0(0)=0$, which implies  $f_0(t)=0$ for all $t$.
\\ In \eqref{DiscreteQuantumb}, the  sums for $C_{22}$ are taken over $(p_1,p_2,p_3,p_4)$ satisfying
\begin{equation}\label{EnergyC22}
p_1 \ + \ p_2 \ = \ p_3\  +\ p_4, \mbox{ and } \ \ \ \  \mathcal{E}(p_1)\ + \ \mathcal{E}(p_2)\ = \ \mathcal{E}(p_3) \ + \mathcal{E}(p_4).
\end{equation}
In this case, unlike in the case of $C_{12}$ and $C_{13}$, we {\it cannot} infer from \eqref{EnergyC22} that there exists a vector $P$  and  $k_1$, $k_2$, $k_3, k_4\geq0$, $k_1, k_2, k_3, k_4\in\mathbb{Z}$ such that 
$$p_1=k_1 P;~~~p_2=k_2 P;~~~p_3=k_3 P;~~~p_4=k_4 P,~~~k_1+k_2=k_3+k_4.$$ 
However, let us consider the following simplified version of \eqref{DiscreteQuantumb} for $C_{22}$ 
\begin{equation}\label{DiscreteC22}\begin{aligned}
\dot{F}_{k_1}
=&\ \mathcal{C}_{22}[F](k_1):=\quad\sum_{\substack{k_1+k_2=k_3+k_4\\k_2,k_3,k_4\in\mathbb{I}}}\mathcal{K}^{13}_{k_1,k_2,k_3,k_4}\{(F_{k_1}+1)(F_{k_2}+1)F_{k_3}F_{k_4}-\\
& -F_{k_1}F_{k_2}(F_{k_3}+1)(F_{k_4}+1)\}, ~\forall k_1\in\mathbb{I}.\end{aligned}
\end{equation}
Recall that $\mathbb{I}=\{1,\cdots,I\}$. We also suppose that $I\geq3$.
We  have the following conservation of energy
\begin{equation}\label{ConservationE1b}
\sum_{k=1}^I k\dot{F_k}=0,
\end{equation}
or equivalently
\begin{equation}\label{ConservationE2b}
\sum_{k=1}^I k{F_k}=\mbox{const}.
\end{equation}
For $C_{22}$, the following ``conservation of mass'' also holds
\begin{equation}\label{ConservationM1}
\sum_{k=1}^I \dot{F_k}=0,
\end{equation}
or equivalently
\begin{equation}\label{ConservationM2}
\sum_{k=1}^I {F_k}=\mbox{const}.
\end{equation}
Similar  to the case of $C_{12}$, define $$G_k=\frac{F_k}{F_k+1},$$
then 
$$F_k=\frac{G_k}{1-G_k},$$
and 
the system  \eqref{DiscreteQuantumb} can be now written
\begin{equation}\label{DiscreteQuantumaConvertedb}\begin{aligned}
&\frac{\dot{G}_{k_1}}{(1-G_{k_1})^2}=\overline{\mathcal{C}_{22}}[G]:=\\
:=&\quad\mathcal{K}^{13}_{k_1,k_2,k_3,k_4}\sum_{\substack{k_1+k_{2}=k_3+k_4,\\
|k_1|+|k_2|=|k_3|+|k_4|
}}\frac{G_{k_3}G_{k_4}-G_{k_1}G_{k_2}}{(1-G_{k_1})(1-G_{k_2})(1-G_{k_3})(1-G_{k_4})}
.\end{aligned}
\end{equation}
This system can be rewritten as
\begin{equation}\begin{aligned}
{\dot{G}}
=&\mbox{diag}\left( {\begin{array}{cc}
  {(1-G_1)^2}\\       \cdots \\   {(1-G_I)^2}  \end{array} } \right)\times\\
 &\times\sum_{\substack{k_1+k_{2}=k_3+k_4,\\
|k_1|+|k_2|=|k_3|+|k_4|
}}\big[K_{\bar{X}_{k_{3}}+ \bar{X}_{k_{4}}\rightarrow \bar{X}_{k_{2}}+\bar{X}_{k_{1}}}(G)\\
&-K_{\bar{X}_{k_{2}}+\bar{X}_{k_{1}}\rightarrow \bar{X}_{k_{3}}+\bar{X}_{k_{4}}}(G)\big](\bar{X}_{k_{1}}+\bar{X}_{k_{2}}-\bar{X}_{k_{3}}-\bar{X}_{k_{4}}).\end{aligned}
\end{equation}
where $\bar{X}_k$ is, the vector 
$$\left( {\begin{array}{cc}
 0\\       \cdots\\ 1\\ \cdots \\   0  \end{array} } \right),$$
in which the only element that is $1$ is the $k$-th one,
and 
$$K_{\bar{X}_{k_3}+ \bar{X}_{k_4}\rightarrow \bar{X}_{k_2}+\bar{X}_{k_{1}} }(G)=\mathcal{K}^{13}_{k_1,k_2,k_3,k_4}\frac{G_{k_1}G_{k_2}}{(1-G_{k_1})(1-G_{k_2})(1-G_{k_3})(1-G_{k_4})},$$
$$K_{\bar{X}_{k_2}+\bar{X}_{k_{1}}\rightarrow \bar{X}_{k_3}+ \bar{X}_{k_4}}(G)=\mathcal{K}^{13}_{k_1,k_2,k_3,k_4}\frac{G_{k_3}G_{k_4}}{(1-G_{k_1})(1-G_{k_2})(1-G_{k_3})(1-G_{k_4})}.$$
We can also write 
\begin{eqnarray}\nonumber
{\dot{G}}
&=&\mbox{diag}\left( {\begin{array}{cc}
  {(1-G_1)^2}\\       \cdots \\   {(1-G_I)^2}  \end{array} } \right)\sum_{y\leftrightarrow y'}\left[K_{y \rightarrow y'}(G)-K_{y' \rightarrow y}(G)\right](y'-y ),
\end{eqnarray}
where $y\leftrightarrow y'$ range over the reversible reactions shown above.
\begin{theorem}\label{TheoremQnm}  
For any initial condition, the solution $$F(t)=(F_k(t))_{k\in\mathbb{I}}$$ of the quantum Boltzmann equation \eqref{DiscreteC22} converges to an equilibrium state $F^*=(F_k^*)_{k\in\mathbb{I}}$, where $$F^*_k=\frac{1}{e^{\rho_2(k-1)-\rho_1(k-2)}-1}.$$  Moreover, the solution $F(t)$ of \eqref{DiscreteC22} converges to $F^*$ exponentially fast in the following sense: there exist positive constants $C_1$, $C_2$ such that $$\max_{k\in\mathbb{I}}|F_k(t)-F_k^*|<C_1e^{-C_2t}.$$
\end{theorem}
\begin{proof}
We set $$G_k^*=\frac{F_k^*}{F_k^*+1}.$$
By the same argument used to obtain \eqref{C12equilibrium}, we deduce that \begin{equation}\label{Cmnequilibrium}
\frac{G^y }{(G^*)^y}
=\frac{G^{y'} }{(G^*)^{y'} }\end{equation}
holds true for all reactions $y \leftrightarrow y',$ which implies $G^y=G^{y'}$ 
since $(G^*)^y=(G^*)^{y'}$ for all reactions $y \leftrightarrow y'.$ In the  case of $C_{22}$, we obtain the relation $G^{*}_{k_1}\cdot G^{*}_{k_2}  = G^{*}_{k_3}\cdot G^*_{k_4}$ for all $k_1$, $k_2$, $k_3$, $k_4$ such that $k_1+k_2=k_3+k_4\le I$. From the relation $G_{k_1}\cdot G_{k_2}  = G_{k_3}\cdot G_{k_4}$ and the fact that $k+(k-2)=2(k-1)$, the following identity holds true 
$$G_k(G_1)^{k-2}=(G_2)^{k-1}.$$
We then obtain $G_k=(G_2)^{k-1}/(G_1)^{k-2}$, which leads to $G_k=e^{\rho'_2(k-1)-\rho'_1(k-2)}$ for some numbers $\rho_1',\rho_2'$.

Since
$$\sum_{k=1}^I\frac{kG_k}{1-G_k}=\sum_{k=1}^I\frac{kG_k^*}{1-G_k^*},$$
we then have
 $$\sum_{k=1}^I\frac{ke^{\rho'_2(k-1)-\rho'_1(k-2)}}{1-e^{\rho'_2(k-1)-\rho'_1(k-2)}}=\sum_{k=1}^I\frac{ke^{\rho_2(k-1)-\rho_1(k-2)}}{1-e^{\rho_2(k-1)-\rho_1(k-2)}}.$$

We can proceed like in \cite{HornJackson:1972:GMA} to obtain   $\tilde{\rho}'=\tilde{\rho}$ and $\bar{\rho}=\bar{\rho}'$, yielding $\rho_1=\rho_1'$ and $\rho_2=\rho_2'$.

%
%
%
%

We  can still use the Petri net argument of \cite{AngeliLeenheerSontag:2011:PRF} or the result in \cite{Craciun:2015:TDI}, to prove that the system is persistent.  For example, to use the method from \cite{AngeliLeenheerSontag:2011:PRF}, we note that we have two {\it siphons}  $\{X_1, X_2,\cdots, X_I\}$, $\{X_2,\cdots, X_I\}$. However, we also have the conservations of mass and energy 
$$\sum_{k=1}F_k=\mbox{constant},$$
$$\sum_{k=1}kF_k=\mbox{constant},$$
that leads to the {\it $P$-semiflow} 
$$\sum_{k=2}(k-1)F_k=\mbox{constant}.$$
Therefore, similar  to the case of $C_{12}$, it follows that the system is persistent, and we can use the same Lyapunov function as in the proof of Theorem \ref{TheoremQ12} to obtain the desired convergence result.

\end{proof}
\begin{remark}
If $I<3$  then $F_{k}^*\equiv 0$. If $I=3$  then $F_{2}^*\equiv 0$ and $F^*_{1}=\frac{1}{e^\rho-1}$, $F^*_{3}=\frac{1}{e^{3\rho}-1}$ for some $\rho=\rho(P_0)$. 
\end{remark}
\section{A reaction network approach for the sum of $C_{12}, C_{22}, C_{13}$}\label{Sec:SumQnm}
Let us consider the following equations
\begin{equation}\label{DiscreteC12C22}
\dot{F}_{k_1}=\mathcal{C}_{12}[F](k_1) \ + \ \mathcal{C}_{22}[F](k_1),
\end{equation}
and
\begin{equation}\label{DiscreteC12C22C13}
\dot{F}_{k_1}=\mathcal{C}_{12}[F](k_1) \ + \ \mathcal{C}_{22}[F](k_1)\ + \ \mathcal{C}_{13}[F](k_1),
\end{equation}
where $\mathcal{C}_{12}$,  $\mathcal{C}_{22}$, $\mathcal{C}_{13}$ are the operators defined in  \eqref{DiscreteC12}, \eqref{DiscreteC13}, \eqref{DiscreteC22}.

The following theorem then follows by exactly the same argument as in Theorem \ref{TheoremQnm}
\begin{theorem}\label{TheoremSumQnm}  For any initial condition, the solution $$F(t)=(F_k(t))_{k\in\mathbb{I}}$$ of the quantum Boltzmann equation \eqref{DiscreteC12C22} or \eqref{DiscreteC12C22C13} converges to an equilibrium state $F^*=(F_k^*)_{k\in\mathbb{I}}$, where $F^*_k=\frac{1}{e^{\rho k}-1}$ for some constant $\rho$.  Moreover, the solution $F(t)$ of \eqref{DiscreteC13} converges to $F^*$ exponentially fast in the following sense: there exists positive constants $C_1$, $C_2$ such that $$\max_{k\in\mathbb{I}}|F_k(t)-F_k^*|<C_1e^{-C_2t}.$$
\end{theorem}

\section{Conclusion}

In this work, we point out a connection between quantum Boltzmann models derived in~\cite{tran2020boltzmann}\ and chemical reaction network models. We prove that  the  discrete, simplified versions of some differential equations for these quantum Boltzmann models relax to an equilibrium point, by a toric dynamical system approach, similar to the one used in a recently proposed proof of the global attractor conjecture~\cite{Craciun:2015:TDI}.   \\

\noindent
{\bf Acknowledgements.} The authors would like to thank the editors for the invitation to submit the paper for the special issue of the journal. They are also indebted to Professor Roland Glowinski for his encouragement.  G. Craciun was supported by NSF grants DMS-1412643 and DMS-1816238, and by a Simons Foundation fellowship.
M.-B. Tran is  funded in part by  the  NSF Grant DMS-1854453, NSF RTG Grant DMS-1840260, SMU URC Grant 2020, Humboldt Fellowship, SMU Dedman College Linking Fellowship, NSF CAREER  DMS-2044626. \\

\noindent
{\bf Conflict of interest.} 
The  authors state that there is no conflict of interest. 

\bibliographystyle{plain}\bibliography{QuantumBoltzmannPhonons,QuantumBoltzmann}

\def\cprime{$'$} 
\begin{thebibliography}{10}

\bibitem{AlonsoGambaBinh}
R.~Alonso, I.~M. Gamba, and M.-B. Tran.
\newblock The cauchy problem and bec stability for the quantum
  boltzmann-condensation system for bosons at very low temperature.
\newblock {\em arXiv preprint arXiv:1609.07467}, 2016.

\bibitem{Anderson:2001:APG}
David~F. Anderson.
\newblock A proof of the global attractor conjecture in the single linkage
  class case.
\newblock {\em SIAM J. Appl. Math.}, 71(4):1487--1508, 2011.

\bibitem{MR2734052}
David~F. Anderson, Gheorghe Craciun, and Thomas~G. Kurtz.
\newblock Product-form stationary distributions for deficiency zero chemical
  reaction networks.
\newblock {\em Bull. Math. Biol.}, 72(8):1947--1970, 2010.

\bibitem{AngeliLeenheerSontag:2011:PRF}
David Angeli, Patrick De~Leenheer, and Eduardo~D. Sontag.
\newblock Persistence results for chemical reaction networks with
  time-dependent kinetics and no global conservation laws.
\newblock {\em SIAM J. Appl. Math.}, 71(1):128--146, 2011.

\bibitem{ArkerydNouri:2012:BCI}
L.~Arkeryd and A.~Nouri.
\newblock Bose condensates in interaction with excitations: a kinetic model.
\newblock {\em Comm. Math. Phys.}, 310(3):765--788, 2012.

\bibitem{ArkerydNouri:AMP:2013}
L.~Arkeryd and A.~Nouri.
\newblock A {M}ilne problem from a {B}ose condensate with excitations.
\newblock {\em Kinet. Relat. Models}, 6(4):671--686, 2013.

\bibitem{ArkerydNouri:2015:BCI}
L.~Arkeryd and A.~Nouri.
\newblock Bose condensates in interaction with excitations: a two-component
  space-dependent model close to equilibrium.
\newblock {\em J. Stat. Phys.}, 160(1):209--238, 2015.

\bibitem{MR2604624}
Murad Banaji and Gheorghe Craciun.
\newblock Graph-theoretic approaches to injectivity and multiple equilibria in
  systems of interacting elements.
\newblock {\em Commun. Math. Sci.}, 7(4):867--900, 2009.

\bibitem{bernhoff2015half}
N.~Bernhoff.
\newblock Half-space problems for a linearized discrete quantum kinetic
  equation.
\newblock {\em Journal of statistical physics}, 159(2):358--379, 2015.

\bibitem{bernhoff2017boundary}
N.~Bernhoff.
\newblock Boundary layers for discrete kinetic models: multicomponent mixtures,
  polyatomic molecules, bimolecular reactions, and quantum kinetic equations.
\newblock {\em Kinetic \& Related Models}, 10(4):925, 2017.

\bibitem{Boltzmann}
L.~Boltzmann.
\newblock Neuer {B}eweis zweier {S}atze uber das {W}armegleichgewicht unter
  mehratomigen {G}as-molekulen.
\newblock {\em Sitzungsberichte der Kaiserlichen Akademie der Wissenschaften in
  Wien.}, (95):153--164, 1887.

\bibitem{cai2018spatially}
S.~Cai and X.~Lu.
\newblock The spatially homogeneous boltzmann equation for bose-einstein
  particles: Rate of strong convergence to equilibrium.
\newblock {\em arXiv preprint arXiv:1808.04038}, 2018.

\bibitem{craciun2019polynomial}
G.~Craciun.
\newblock Polynomial dynamical systems, reaction networks, and toric
  differential inclusions.
\newblock {\em SIAM Journal on Applied Algebra and Geometry}, 3(1):87--106,
  2019.

\bibitem{Craciun:2015:TDI}
Gheorghe Craciun.
\newblock Toric {D}ifferential {I}nclusions and a proof of the {G}lobal
  {A}ttractor {C}onjecture.
\newblock {\em Submitted}.

\bibitem{MR2561288}
Gheorghe Craciun, Alicia Dickenstein, Anne Shiu, and Bernd Sturmfels.
\newblock Toric dynamical systems.
\newblock {\em J. Symbolic Comput.}, 44(11):1551--1565, 2009.

\bibitem{CraciunFeinberg:2005:MEI}
Gheorghe Craciun and Martin Feinberg.
\newblock Multiple equilibria in complex chemical reaction networks. {I}. {T}he
  injectivity property.
\newblock {\em SIAM J. Appl. Math.}, 65(5):1526--1546 (electronic), 2005.

\bibitem{CraciunNazarovPantea:2013:PAP}
Gheorghe Craciun, Fedor Nazarov, and Casian Pantea.
\newblock Persistence and permanence of mass-action and power-law dynamical
  systems.
\newblock {\em SIAM J. Appl. Math.}, 73(1):305--329, 2013.

\bibitem{Binh9}
M.~Escobedo and M.-B. Tran.
\newblock Convergence to equilibrium of a linearized quantum {B}oltzmann
  equation for bosons at very low temperature.
\newblock {\em Kinetic and Related Models}, 8(3):493--531, 2015.

\bibitem{EscobedoVelazquez:2015:FTB}
M.~Escobedo and J.~J.~L. Vel{\'a}zquez.
\newblock Finite time blow-up and condensation for the bosonic {N}ordheim
  equation.
\newblock {\em Invent. Math.}, 200(3):761--847, 2015.

\bibitem{Feinberglecture}
M.~Feinberg.
\newblock Lectures on chemical reaction networks.
\newblock {\em written version of lectures given at the Mathematical Research
  Center, University of Wisconsin, Madison WI, 1979. Available at
  http://www.crnt.osu.edu/LecturesOnReactionNetworks}.

\bibitem{Feinberg:1972:CBI}
Martin Feinberg.
\newblock Complex balancing in general kinetic systems.
\newblock {\em Arch. Rational Mech. Anal.}, 49:187--194, 1972/73.

\bibitem{Feinberg:1995:TEA}
Martin Feinberg.
\newblock The existence and uniqueness of steady states for a class of chemical
  reaction networks.
\newblock {\em Arch. Rational Mech. Anal.}, 132(4):311--370, 1995.

\bibitem{germain2020optimal}
P.~Germain, A.~D. Ionescu, and M.-B. Tran.
\newblock Optimal local well-posedness theory for the kinetic wave equation.
\newblock {\em Journal of Functional Analysis}, 279(4):108570, 2020.

\bibitem{MR3199409}
Manoj Gopalkrishnan, Ezra Miller, and Anne Shiu.
\newblock A geometric approach to the global attractor conjecture.
\newblock {\em SIAM J. Appl. Dyn. Syst.}, 13(2):758--797, 2014.

\bibitem{GriffinNikuniZaremba:2009:BCG}
A.~Griffin, T.~Nikuni, and E.~Zaremba.
\newblock Bose-condensed gases at finite temperatures.
\newblock {\em Cambridge University Press, Cambridge, 2009}.

\bibitem{GriffinNikuniZaremba:BCG:2009}
Allan Griffin, Tetsuro Nikuni, and Eugene Zaremba.
\newblock {\em Bose-condensed gases at finite temperatures}.
\newblock Cambridge University Press, Cambridge, 2009.

\bibitem{Gunawardena}
J.~Gunawardena.
\newblock Chemical reaction network theory for in-silico biologists.
\newblock {\em Lecture notes available online at http://vcp.med.harvard.edu/
  papers.html, 2003.}

\bibitem{ReichlGust:2012:CII}
E.~D Gust and L.~E. Reichl.
\newblock Collision integrals in the kinetic equations ofdilute
  {B}ose-{E}instein condensates.
\newblock {\em arXiv:1202.3418}, 2012.

\bibitem{ReichlGust:2013:RRA}
E.~D Gust and L.~E. Reichl.
\newblock Relaxation rates and collision integrals for {B}ose-{E}instein
  condensates.
\newblock {\em Phys. Rev. A}, 170:43--59, 2013.

\bibitem{Horn:1972:NAS}
F.~Horn.
\newblock Necessary and sufficient conditions for complex balancing in chemical
  kinetics.
\newblock {\em Arch. Rational Mech. Anal.}, 49:172--186, 1972/73.

\bibitem{Horn:1974:TDO}
F.~Horn.
\newblock The dynamics of open reaction systems.
\newblock In {\em Mathematical aspects of chemical and biochemical problems and
  quantum chemistry ({P}roc. {SIAM}-{AMS} {S}ympos. {A}ppl. {M}ath., {N}ew
  {Y}ork, 1974)}, pages 125--137. SIAM--AMS Proceedings, Vol. VIII. Amer. Math.
  Soc., Providence, R.I., 1974.

\bibitem{HornJackson:1972:GMA}
F.~Horn and R.~Jackson.
\newblock General mass action kinetics.
\newblock {\em Arch. Rational Mech. Anal.}, 47:81--116, 1972.

\bibitem{JinBinh}
S.~Jin and M.-B. Tran.
\newblock Quantum hydrodynamic approximations to the finite temperature trapped
  bose gases.
\newblock {\em Physica D: Nonlinear Phenomena}, 380:45--57, 2018.

\bibitem{KD1}
T.~R. Kirkpatrick and J.~R. Dorfman.
\newblock Transport theory for a weakly interacting condensed {B}ose gas.
\newblock {\em Phys. Rev. A (3)}, 28(4):2576--2579, 1983.

\bibitem{KD2}
T.~R. Kirkpatrick and J.~R. Dorfman.
\newblock Transport in a dilute but condensed nonideal {B}ose gas: Kinetic
  equations.
\newblock {\em J. Low Temp. Phys.}, 58:301--331, 1985.

\bibitem{Lu:2004:OID}
Xuguang Lu.
\newblock On isotropic distributional solutions to the {B}oltzmann equation for
  {B}ose-{E}instein particles.
\newblock {\em J. Statist. Phys.}, 116(5-6):1597--1649, 2004.

\bibitem{Lu:2005:TBE}
Xuguang Lu.
\newblock The {B}oltzmann equation for {B}ose-{E}instein particles: velocity
  concentration and convergence to equilibrium.
\newblock {\em J. Stat. Phys.}, 119(5-6):1027--1067, 2005.

\bibitem{Lu:2013:TBE}
Xuguang Lu.
\newblock The {B}oltzmann equation for {B}ose-{E}instein particles:
  condensation in finite time.
\newblock {\em J. Stat. Phys.}, 150(6):1138--1176, 2013.

\bibitem{ToanBinh}
T.~T. Nguyen and M.-B. Tran.
\newblock Uniform in time lower bound for solutions to a quantum boltzmann
  equation of bosons.
\newblock {\em Archive for Rational Mechanics and Analysis}, 231:63--89, 2019.

\bibitem{Nordheim}
L.W. Nordheim.
\newblock Transport phenomena in einstein-bose and fermi-dirac gases.
\newblock {\em Proc. Roy. Soc. London A}, 119:689, 1928.

\bibitem{pomeau2019statistical}
Y.~Pomeau and M.-B. Tran.
\newblock {\em Statistical Physics of Non Equilibrium Quantum Phenomena},
  volume 967.
\newblock Springer Nature, 2019.

\bibitem{ReichlBook}
L.~E. Reichl.
\newblock {\em A modern course in statistical physics}.
\newblock A Wiley-Interscience Publication. John Wiley \& Sons, Inc., New York,
  fourth edition, 2016.

\bibitem{ReichlGust:2013:TTF}
L.~E. Reichl and E.~D Gust.
\newblock Transport theory for a dilute {B}ose-{E}instein condensate.
\newblock {\em J Low Temp Phys}, 88:053603, 2013.

\bibitem{reichl2019kinetic}
L.~E. Reichl and M.-B. Tran.
\newblock A kinetic equation for ultra-low temperature bose--einstein
  condensates.
\newblock {\em Journal of Physics A: Mathematical and Theoretical},
  52(6):063001, 2019.

\bibitem{SofferBinh2}
A.~Soffer and M.-B. Tran.
\newblock On coupling kinetic and schrodinger equations.
\newblock {\em Journal of Differential Equations}, 265(5):2243--2279, 2018.

\bibitem{SofferBinh1}
A.~Soffer and M.-B. Tran.
\newblock On the dynamics of finite temperature trapped bose gases.
\newblock {\em Advances in Mathematics}, 325:533--607, 2018.

\bibitem{Spohn:2010:KOT}
H.~Spohn.
\newblock Kinetics of the {B}ose-{E}instein condensation.
\newblock {\em Physica D}, 239:627--634, 2010.

\bibitem{CraciunSmithBoldyrevBinh}
M.-B. Tran, G.~Craciun, L.~M. Smith, and S.~Boldyrev.
\newblock A reaction network approach to the theory of acoustic wave
  turbulence.
\newblock {\em Journal of Differential Equations Volume 269, Issue 5, Pages
  4332-4352}, 2020.

\bibitem{tran2020boltzmann}
M.-B. Tran and Y.~Pomeau.
\newblock Boltzmann-type collision operators for bogoliubov excitations of
  bose-einstein condensates: A unified framework.
\newblock {\em Physical Review E}, 101(3):032119, 2020.

\bibitem{tran2020boltzmann2}
M.-B. Tran and Y.~Pomeau.
\newblock On a thermal cloud - {B}ose-{E}instein condensate coupling system.
\newblock {\em European Physical Journal Plus}, 136, 502, 2021.

\bibitem{UehlingUhlenbeck}
E.A. Uehling and G.E. Uhlenbeck.
\newblock On the kinetic methods in the new statistics and its applications in
  the electron theory of conductivity.
\newblock {\em I Phys. Rev.}, 43:552--561, 1933.

\bibitem{PomeauBrachetMetensRica}
S.~M'etens Y.~Pomeau, M.A.~Brachet and S.~Rica.
\newblock Th\'eorie cin\'etique d'un gaz de {B}ose dilu\'e avec condensat.
\newblock {\em C. R. Acad. Sci. Paris S'er. IIb M'ec. Phys. Astr.},
  327:791--798, 1999.

\bibitem{yu2018mathematical}
P.~Y. Yu and G.~Craciun.
\newblock Mathematical analysis of chemical reaction systems.
\newblock {\em Israel Journal of Chemistry}, 58(6-7):733--741, 2018.

\bibitem{ZarembaNikuniGriffin:1999:DOT}
E.~Zaremba, T.~Nikuni, and A.~Griffin.
\newblock Dynamics of trapped {B}ose gases at finite temperatures.
\newblock {\em J. Low Temp. Phys.}, 116:277--345, 1999.

\end{thebibliography}
\end{document}